\documentclass[11pt]{amsart}

\usepackage{epigamath}

\usepackage[english]{babel}

\numberwithin{equation}{section}

\usepackage{enumitem}

\usepackage{babel}
\usepackage{amssymb}
\usepackage{newunicodechar}
\usepackage{mathtools}
\usepackage{varioref}
\usepackage[arrow,curve,matrix]{xy}

\usepackage{colortbl}
\usepackage{graphicx}
\usepackage{tikz}
\usepackage{tikz-cd}

\usepackage{mathrsfs}

\usepackage{multirow}
\usepackage{multicol}
\usepackage{longtable} 
\usepackage{caption}
\newcolumntype{M}[1]{>{\centering\arraybackslash}m{#1}} 

\DeclareFontFamily{OMS}{rsfs}{\skewchar\font'60}
\DeclareFontShape{OMS}{rsfs}{m}{n}{<-5>rsfs5 <5-7>rsfs7 <7->rsfs10 }{}
\DeclareSymbolFont{rsfs}{OMS}{rsfs}{m}{n}
\DeclareSymbolFontAlphabet{\scr}{rsfs}
\DeclareSymbolFontAlphabet{\scr}{rsfs}

\theoremstyle{plain}

\newtheorem{thm}{Theorem}[section]
\newtheorem{lem}[thm]{Lemma}
\newtheorem{prop}[thm]{Proposition}

\DeclareMathOperator{\Pic}{Pic}

\DeclareMathOperator{\lcm}{lcm}
\DeclareMathOperator{\Cl}{Cl}
\DeclareMathOperator{\ct}{ct}

\newcommand{\sF}{\scr{F}}

\newcommand{\sM}{\scr{M}}
\newcommand{\sN}{\scr{N}}
\newcommand{\sO}{\scr{O}}

\newcommand{\sQ}{\scr{Q}}

\newcommand{\sS}{\scr{S}}
\newcommand{\sT}{\scr{T}}

\newcommand{\cR}{\mathcal R}

\newcommand{\bC}{\mathbb{C}}

\newcommand{\bP}{\mathbb{P}}
\newcommand{\bQ}{\mathbb{Q}}

\newcommand{\bZ}{\mathbb{Z}}

\newcommand{\bfB}{\mathbf{B}}

\setlist[enumerate]{label={\rm(\thethm.\arabic*)}, ref=\thethm.\arabic*}

\EpigaVolumeYear{9}{2025} \EpigaArticleNr{12} \ReceivedOn{March 4, 2024}
\InFinalFormOn{October 19, 2024}
\AcceptedOn{November 20, 2024}

\title{Kawamata--Miyaoka-type inequality for $\bQ$-Fano varieties with canonical singularities II: Terminal $\bQ$-Fano threefolds}
\titlemark{Kawamata--Miyaoka-type inequality for terminal $\bQ$-Fano threefolds}

\author{Haidong Liu}
\address{Sun Yat-Sen University, Department of Mathematics, Guangzhou, 510275, China}
\email{liuhd35@mail.sysu.edu.cn, 
jiuguiaqi@gmail.com}
\author{Jie Liu}
\address{Institute of Mathematics, Academy of Mathematics and Systems Science, Chinese Academy of Sciences, Beijing, 100190, China}
\email{jliu@amss.ac.cn}

\authormark{H.~Liu and J.~Liu}

\AbstractInEnglish{We prove an optimal Kawamata--Miyaoka-type inequality for terminal $\bQ$-Fano threefolds with Fano index at least $3$. As an application, any terminal $\bQ$-Fano threefold $X$ satisfies the following Kawamata--Miyaoka-type inequality: 
		\[
		c_1(X)^3 < 3c_2(X)c_1(X).
		\]}


\acknowledgement{H.~Liu is supported in part by the National Key Research and Development Program of China (No.~2023YFA100980001). J.~Liu is supported by the National Key Research and Development Program of China (No.~2021YFA1002300), the CAS Project for Young Scientists in Basic Research (No.~YSBR-033), the NSFC grant (No.~12288201) and the Youth Innovation Promotion Association CAS.}

\begin{document}



\maketitle

\begin{prelims}

\DisplayAbstractInEnglish

\bigskip

\DisplayKeyWords

\medskip

\DisplayMSCclass

\end{prelims}


\newpage

\setcounter{tocdepth}{1}

\tableofcontents


\section{Introduction}
	
	A normal projective variety is called \emph{Fano} if its anti-canonical divisor is $\bQ$-Cartier and ample. A Fano variety is called \emph{$\mathbb Q$-Fano} if it is $\mathbb Q$-factorial and its Picard number is one. According to the minimal model program, $\bQ$-Fano varieties with terminal singularities are one of the building blocks of algebraic varieties, and they form a bounded family (see \cite{Kawamata1992,KollarMiyaokaMoriTakagi2000,Birkar2021}).
	
	Thanks to Reid's orbifold Riemann--Roch formula, see~\cite{Reid1987}, terminal $\bQ$-Fano threefolds have been studied intensively in the past two decades; see \cite{AltinokBrownReid2002,Suzuki2004,BrownSuzuki2007a,BrownSuzuki2007,Prokhorov2010,Prokhorov2013,ChenJiang2016,BrownKasprzyk2022,Prokhorov2022a,Prokhorov2022} and the references therein. In particular, all the possibilities for their numerical types are obtained in the \emph{Graded Ring Database} (\textsc{Grdb} for short, \cite{BrownKasprzyk2009}). 
	
	One of the key ingredients in the computation of possible numerical types of terminal $\bQ$-Fano threefolds is an effective version of the \emph{Kawamata--Miyaoka-type inequality} (see \cite[Proposition 1]{Kawamata1992}, \cite[Theorem~1.7]{Suzuki2004} and \cite[Theorem 2.6]{BrownKasprzyk2022} for more details). In our previous paper \cite{LiuLiu2023}, we established such an effective Kawamata--Miyaoka-type inequality for terminal $\bQ$-Fano varieties in arbitrary dimension. In particular, using a refined argument and Reid's orbifold Riemann--Roch formula, we have proved the following inequality in dimension three, which significantly improves the known ones in the literature.
	
	\begin{thm}[\textit{cf.} \protect{\cite[Theorem 1.2]{LiuLiu2023}}]
		\label{t.LiuLiu3folds}
		Let $X$ be a terminal $\bQ$-Fano threefold. Then we have
		\[
		c_1(X)^3 \leq  \frac{25}{8} c_2(X)c_1(X).
		\]
	\end{thm}
	
	In this paper, we aim to continue in this direction. To give a precise statement, let us recall the definition of Fano indices. In general, let $X$ be a singular Fano variety. We can define the Fano index of $X$ in two different ways:
	\begin{equation*}
		\begin{split}
			q\mathbb Q(X) &\coloneqq\max\{q \mid -K_X \sim_{\mathbb Q}qA, \; A\in \Cl X \},\\
			qW(X) &\coloneqq\max\{q \mid -K_X \sim qB, \; B\in \Cl X \}.
		\end{split}
	\end{equation*}
	If we assume in addition that $X$ has at worst terminal singularities, then the Weil divisor class group $\Cl(X)$ is finitely generated and the numerical equivalence for $\bQ$-Cartier Weil divisors coincides with the $\bQ$-linear equivalence. This implies that both $q\mathbb Q(X)$ and $qW(X)$ are positive integers. 
	
	The main result of this paper is the following optimal Kawamata--Miyaoka-type inequality for terminal $\bQ$-Fano threefolds with large $q\bQ$.
	
	\begin{thm}[\textit{cf.} Lemma~\ref{c.QFano3foldscandidates}, Theorems~\ref{t.qQ=5} and~\ref{t.qQ=4}]
		\label{t.main-thm}
		Let $X$ be a terminal $\bQ$-Fano threefold such that $q\bQ(X)\geq 3$. Then we have
		\[
		c_1(X)^3 \leq \frac{121}{41} c_2(X)c_1(X),
		\]
		and  equality holds if and only if\, $X$ is isomorphic to $\bP(1,2,3,5)$.
	\end{thm}
	
	As an immediate application, combining Theorem~\ref{t.main-thm} with \cite[Theorem 4.4]{LiuLiu2023} yields the following generalisation of Theorem~\ref{t.LiuLiu3folds}, which confirms a recent conjecture of K.~Suzuki \cite[Conjecture 1]{Suzuki2024}. 
	
	\begin{thm}\label{thm.b3}
		Let $X$ be a terminal $\bQ$-Fano threefold. Then we have
		\[
		c_1(X)^3 < 3c_2(X)c_1(X).
		\]
	\end{thm}
	
	The main idea of the proof of Theorem~\ref{t.main-thm} is as follows. Let $X$ be a terminal $\bQ$-Fano threefold with $q\bQ(X)\geq 3$ which does not satisfy the inequality in Theorem~\ref{t.main-thm}. Using the orbifold Riemann--Roch theorem, \cite[Corollary 4.3]{LiuLiu2023} and \cite[Proposition 3.2]{Prokhorov2024}, we have the following two numerical possibilities for $X$ (see Lemma~\ref{c.QFano3foldscandidates}):
	\begin{enumerate}
		\item $qW(X)=q\bQ(X)=4$ and $\cR_X=\{7,13\}$,
		
		\item $qW(X)=q\bQ(X)=5$ and $\cR_X=\{3,7^2\}$,
	\end{enumerate}
	where $\cR_X$ denotes the \emph{local index basket} of $X$ (see Section~\ref{sub.RR-formula}). The main contribution of this paper is to exclude these two cases. To this end, we will introduce two different approaches. In the first one we use the theory of foliations to rule out the second case; it is relatively simple (see Section~\ref{s.qbQ=5}).  The first case is much more difficult; we will follow a second approach developed by Yu.~Prokhorov in \cite{Prokhorov2010} by using the Sarkisov link (see Section~\ref{s.qbQ=4}).
        
	Theorem~\ref{thm.b3} can be directly applied to show that 13559 numerical types in the \textsc{Grdb} \cite{BrownKasprzyk2009} do not actually occur for terminal $\bQ$-Fano threefolds and the Hilbert series of a terminal $\bQ$-Fano threefold always lies in the list $\mathcal{F}_{\mathrm{ss}}$; see \cite[Theorem 1.2]{BrownKasprzyk2022}. On the other hand, if $q\bQ(X)\geq 9$, then \cite[Section~4]{BrownSuzuki2007a} and \cite{Prokhorov2010} show that there are exactly nine numerical types which can be geometrically realised (\textnumero\,23--31 in Table~\ref{table.qgeq6}). However, the situation becomes much more complicated for smaller $q\bQ$. For instance, by Theorem~\ref{thm.b3}, if $q\bQ(X)=8$, there are four possibilities (\textnumero\,19--22 in Table~\ref{table.qgeq6}) for the numerical type of $X$ with  corresponding local index baskets as follows:
	\[
	\{3^2,5\},\quad \{3,7\},\quad\{5,7\}, \quad\{3,5,11\}.
	\]
	It is known that the first three cases can be geometrically realised by appropriate examples, see \cite[Table 1]{BrownSuzuki2007a}, and we will treat the last case in Section~\ref{s.qQ=8} (see \textnumero\,22 in Table~\ref{table.qgeq6} and Theorem~\ref{t.qQ=8}), from which we derive the following result.
	
	\begin{thm}
		Let $X$ be a terminal $\bQ$-Fano threefold with $q\bQ(X)=8$. Then its local index basket $\cR_X$  cannot be $\{3,5,11\}$.
	\end{thm}
	
	Finally, we remark that Theorem~\ref{thm.b3} cannot be generalised to $\bQ$-Fano threefolds with canonical singularities as $\bP(1,3,7,11)$ shows (see \cite[Section~4.5]{BrownKasprzyk2022}), and our proof of Theorem~\ref{t.main-thm} is not completely independent of previously known non-existence results in the literature (see the proof of Lemma~\ref{c.QFano3foldscandidates} for more details). Moreover, there are also some non-existence results, which cannot be recovered by Theorem~\ref{t.main-thm} (\textit{e.g.}~\textnumero\,5 and \textnumero\,8 in \cite[Lemma 8.1]{Prokhorov2013} and \textnumero\,5, \textnumero\,8 and \textnumero\,9 in \cite[Lemma 9.1]{Prokhorov2013}).
	
	\subsection*{Acknowledgements} 
	We would like to thank Yu.~Prokhorov and C.~Jiang for helpful communications and suggestions. In particular, we thank Yu.~Prokhorov for allowing us to include Table~\ref{table.qWnot=qQ=4} in our appendix. We thank the referees for the detailed remarks and suggestions, which helped us to significantly improve the exposition of this paper and also to simplify some arguments.

	\section{Preliminaries}
	
	Throughout this paper, we work over $\bC$, and varieties are always supposed to be irreducible. We will freely use the terminology of \cite{KollarMori1998} for birational geometry, especially the minimal model program (MMP for short). 
	
	\subsection{Orbifold Riemann--Roch formula}\label{sub.RR-formula}
	
	Let $X$ be a terminal Fano threefold and $q\coloneqq qW(X)$. 
	According to  Reid~\cite{Reid1987}, there is a basket of orbifold points
	\[
	\bfB_X=\left\{\left(r_{i}, b_{i}\right) \mid i=1, \ldots, s ;\, 0<b_{i}\leq\frac{r_{i}}{2} ;\, b_{i} \text{ is coprime to } r_{i}\right\}
	\]
	associated to $X$, where a pair $\left(r_{i},b_{i} \right)$ corresponds to an orbifold point $Q_{i}$ of type $\frac{1}{r_{i}}\left(1, -1, b_{i}\right)$. 
	Denote by $\mathcal{R}_X$ the collection of $r_i$ (permitting weights) appearing in $\bfB_X$, and simply write it as a set of integers whose weights appear in superscripts, say for example 
	\[
	\mathcal{R}_X=\{3,7,7\}=\{3,7^2\}.
	\]
	Note that $r_X \coloneqq \lcm (\mathcal R_X)$ coincides with the Gorenstein index of $X$. 
	Let $A$ be a $\mathbb Q$-Cartier Weil divisor on $X$ such that $-K_X\sim qA$. According to \cite[Corollary~10.3]{Reid1987},
	\begin{equation}\label{eq.reid}
		\chi(tA)=1+\frac{t(q+t)(q+2t)}{12}A^3+\frac{t}{12q}c_2(X)c_1(X) +\sum_{Q\in \bfB_X}c_Q(tA)
	\end{equation}
	for $t\in \bZ$; here if the orbifold point $Q$ is of type $\frac{1}{r_Q}(1,-1,b_Q)$ and $0\leq i_{Q,t}<r_Q$ is the integer uniquely determined by $qi_{Q,t}\equiv -t \mod r_Q$, then
	\[
	c_Q(tA)=-\frac{i_{Q,t}(r^2_Q-1)}{12r_Q}+\sum_{j=0}^{i_{Q,t}-1}\frac{\overline{jb_Q}\left(r_Q-\overline{jb_Q}\right)}{2r_Q},
	\] where the symbol $\overline{\bullet}$ means the smallest residue mod $r_Q$ and $\sum_{j=0}^{-1}\coloneqq0$. 
	
	If $t=-q$, then \eqref{eq.reid} and  Serre's duality imply that
	\begin{equation}\label{eq.range}
		1=\frac{1}{24}c_2(X)c_1(X) +\frac{1}{24}\sum_{Q\in \bfB_X} \left(r_Q-\frac{1}{r_Q}\right).
	\end{equation}
	
	If $q\geq 3$ and $-q <t<0$, then $\chi(tA)=0$ by the Kawamata--Viehweg vanishing theorem. Hence \eqref{eq.reid} implies that
	\begin{equation}\label{eq.chita}
		1+\frac{t(q+t)(q+2t)}{12}A^3+\frac{t}{12q}c_2(X)c_1(X) +\sum_{Q\in \bfB_X}c_Q(tA)=0
	\end{equation}
	for $-q<t<0$.
	In particular, if $q\geq 3$ and $t=-1$, then we obtain 
	\begin{equation}\label{eq.primvolume}
		A^3=\frac{12}{(q-1)(q-2)}\left(1-\frac{ c_2(X)c_1(X)}{12q}+\sum_{Q\in \bfB_X}c_Q(-A)\right). 
	\end{equation}
	Moreover, the degree $A^3$ and the Gorenstein index $r_X$ have the following relations (see \cite[Lemma 1.2]{Suzuki2004} for example):  
	\begin{equation}\label{eq.coprimeandint}
		\gcd(r_X, q)=1\quad \text{and}\quad  r_X\cdot A^3\in \mathbb Z_{> 0}.
	\end{equation}
	
	\subsection{Fano indices}
	
	We collect in the following some basic facts about the Fano indices of terminal Fano threefolds. We refer the reader to \cite{Suzuki2004}, \cite[Lemmas 3.2 and~3.3]{Prokhorov2010} and \cite[Proposition 3.3]{Prokhorov2022} for more details.
	
	\begin{lem}
		\label{l.indexQFano}
		Let $X$ be a terminal Fano threefold.
		\begin{enumerate}
		\item\label{e.qW|qQ} We have $qW(X)\mid q\bQ(X)$.
			
			\item\label{e.qQdex} If\, $-K_X \sim_{\bQ} q A$ for some Weil divisor $A$, then $q\mid q\bQ(X)$.
			
			\item\label{e.qWinddex} If\, $-K_X\sim q B$ for some Weil divisor $B$, then $q\mid qW(X)$.
			
			\item If\, $q\bQ(X)\geq 5$, then $qW(X)=q\bQ(X)$.
			
			\item We have $qW(X),q\bQ(X)\in \{1,\dots,9,11,13,17,19\}$.
		\end{enumerate}
	\end{lem}
	\begin{proof}
		The first three statements follow from \cite[Lemma 3.2]{Prokhorov2010}.
		For the last two statements, without loss of generality we may assume that $q\bQ(X)\geq 5$ and let $A$ be a $\bQ$-Cartier Weil divisor such that $-K_X\sim_{\bQ} q\bQ(X)A$. Let $\mu\colon X'\rightarrow X$ be a small $\bQ$-factorialization. Then we have $-K_{X'}\sim_{\bQ} q\bQ(X) A'$, where $A'=\mu^*A$ is a nef and big $\bQ$-Cartier Weil divisor. Running a MMP $g\colon X'\dashrightarrow X''$ yields a Mori fibre space $f\colon X''\rightarrow Z$. Denote by $A''$ the strict transform of $A'$ on $X''$. Then we obtain $-K_{X^{''}}\sim_{\bQ} q\bQ(X) A''$. If $\dim(Z)> 0$, then the general fibre $F$ of $f$ is a non-singular Fano variety such that $-K_F\sim_{\bQ} q\bQ(X) A''|_F$, which is impossible as $\dim F\leq 2$ and $A''|_F$ is Cartier. So $Z$ is a point and $X''$ is a terminal $\bQ$-Fano threefold. 
		
		By item~\eqref{e.qQdex}, we have $q\bQ(X)\mid q\bQ(X'')$, and then it follows from \cite[Theorem 1.4]{Prokhorov2010} and \cite[Proposition 3.3]{Prokhorov2022} that we must have
		\[
		q\bQ(X)=qW(X'')=q\bQ(X'')\in \{5,\dots,9,11,13,17,19\}.
		\]
		Let $m$ be the order of $K_{X''}+q\bQ(X)A''$. By \cite[Proposition 3.4]{Prokhorov2022}, we have $(m,q\bQ(X''))=1$. On the other hand, by the negativity lemma, we also have
		$		g^*(K_{X''}+q\bQ(X)A'') = K_{X'}+q\bQ(X)A'
		$.
		So we obtain $m(K_{X'}+q\bQ(X)A')\sim 0$ and then $m(K_X+q\bQ(X)A)\sim 0$. Thus it follows from \cite[Lem\-ma~3.2(iv)]{Prokhorov2010} that $qW(X)=q\bQ(X)$.
	\end{proof}

	\subsection{Algorithm and numerical types}\label{sub.ac}
	
	Let $X$ be a terminal Fano threefold. We can define a positive rational number $b_X$ as follows (see \cite[Corollary 7.3]{IwaiJiangLiu2023}):
	\[
	b_X\coloneqq \frac{c_1(X)^3}{c_2(X)c_1(X)}.
	\]
	Now with the help of a computer program or using the \textsc{Grdb} \cite{BrownKasprzyk2009}, we get the following result, which is our starting point~-- see also \cite[Proof of Theorem 4.4 and Remark 4.5]{LiuLiu2023}.

	\begin{lem}
		\label{l.Fanothreefoldscandidates}
        \label{c.QFano3foldscandidates}
		Let $X$ be a terminal $\bQ$-Fano threefold such that $q\bQ(X)\geq 3$ and $b_X\geq 121/41$.  Then $\Cl(X)$ is torsion-free, and one of the following cases holds:
		\begin{enumerate}          
			\item\label{c.r.4} $qW(X)=q\bQ(X)=4$, $\cR_X=\{7,13\}$ and $b_X=3$ $($cf. \cite[\textnumero\,41313]{BrownKasprzyk2009}$)$;
			
			\item\label{c.r.5} $qW(X)=q\bQ(X)=5$, $\cR_X=\{3,7^2\}$ and $b_X=25/8$ $($cf. \cite[\textnumero\,41436]{BrownKasprzyk2009}$)$;
			
			\item $X\cong \bP(1,2,3,5)$ and $b_X=121/41$ $($cf. \cite[\textnumero\,41510]{BrownKasprzyk2009}$)$.
		\end{enumerate}
	\end{lem}
	
	\begin{proof}
		By \cite[Proposition 3.2]{Prokhorov2024} and Table~\ref{table.qWnot=qQ=4}, we may assume that $qW(X)=q\bQ(X)$. Moreover, by \cite[Corollary 4.3 and Theorem 4.4]{LiuLiu2023}, we have 
	        \begin{equation}
			\tag{$\clubsuit$} 
			\label{e.b_Xbounds}
			 b_X
			\begin{dcases}
				\leq \frac{64}{21} & \textup{if } qW(X)=4, \\[4pt]
				\leq \frac{25}{8}  & \textup{if } qW(X)=5, \\[4pt]
				<3         & \textup{otherwise}.
			\end{dcases}
		\end{equation}
        Then we use a computer program written in Python, whose algorithm is sketched as follows. 
		
		\textit{Step 1.} As $c_2(X)c_1(X)>0$ by \cite[Proposition 1]{Kawamata1992}, we can list a huge but finite number of possibilities for $\cR_X$ and $c_2(X)c_1(X)$ satisfying \eqref{eq.range}. 
		
		\textit{Step 2.} For each $q=qW(X)\geq 3$, we calculate $A^3$ by \eqref{eq.primvolume} and pick up those satisfying \eqref{eq.chita} and \eqref{eq.coprimeandint}.
		
		\textit{Step 3.} We find all candidates satisfying \eqref{e.b_Xbounds} and $b_X\geq 121/41$, which are as follows:
  
    \renewcommand*{\arraystretch}{1.6}
	\begin{longtable*}{|M{0.7cm}|M{1cm}|M{3cm}|M{2cm}|M{4cm}|}
		\hline
		
		\textnumero
        & $q$
		& $\cR_X$
		& $b_X$
        & \cite{BrownKasprzyk2009}
		\\
		\hline

        1
        & 4
		& $\{7,13\}$
		& $3$
        & \textnumero\,41313
		\\
		\hline

        2
        & 5
		& $\{3,7^2\}$
		& $\frac{25}{8}$
        & \textnumero\,41436
		\\
		\hline

        3
        & 5
		& $\{4,7\}$
		& $3$
        & \textnumero\,41449
		\\
		\hline
		
		4
        & 11
		& $\{2,3,5\}$
		& $\frac{121}{41}$
        & \textnumero\,41510
		\\
		\hline
	\end{longtable*}
  
        Case \textnumero\,3 was excluded in \cite[Section~7.5]{Prokhorov2013}.
        If the numerical type of $X$ appears as \textnumero\,4, then we have $\dim|-K_X|=23$, \textit{cf.}  \cite[\textnumero\,41510]{BrownKasprzyk2009}, and hence $X\cong \bP(1,2,3,5)$ by \cite[Theorem 1.4]{Prokhorov2010}. Finally, the torsion-freeness of $\Cl(X)$ follows from \cite[Proposition 2.9]{Prokhorov2010}.
	\end{proof}
	
	\subsection{Torsion part of $\boldsymbol{\Cl(X)}$}
	
	Let $X$ be a terminal $\bQ$-Fano threefold. If $q\bQ(X)\geq 8$, then we have $qW(X)=q\bQ(X)\geq 8$ by Lemma~\ref{l.indexQFano}, and so $\Cl(X)$ is torsion-free by \cite[Lemma 3.5]{Prokhorov2010}. On the other hand, if $5\leq q\bQ(X)\leq 7$ and the torsion subgroup $\Cl(X)_t$ of $\Cl(X)$ is non-trivial, then the possibilities for the numerical types of $X$ have been obtained in \cite[Proposition 3.4]{Prokhorov2022}, which can be refined as follows.
	
	\begin{prop}
		\label{p.torsion}
		Let $X$ be a terminal $\bQ$-Fano threefold such that $q\bQ(X)\geq 5$ and $\Cl(X)_t$ is non-trivial. Then $\Cl(X)_t$ is cyclic of order $\iota$. Moreover, let $T$ be a generator of $\Cl(X)_t$, and let $A$ be a Weil divisor such that $-K_X\sim qW(X)A$. Then one of the  cases in Table~\ref{table.qgeq5} holds. 
  
		\footnotesize
		\renewcommand*{\arraystretch}{1.6}
		\begin{longtable}
			{|M{0.3cm}|M{0.4cm}|M{0.2cm}|M{1.6cm}|M{0.4cm}|M{0.6cm}|M{0.6cm}|M{0.6cm}|M{1cm}|M{1.2cm}|M{1.2cm}|M{1.2cm}|}
			
			\caption{$q\bQ(X)\geq 5$ with $\Cl(X)_t\not=\{e\}$}
			\label{table.qgeq5}
			\\
			\hline
			
			\multirow{2}{*}{\textnumero}
			& \multirow{2}{*}{$q\bQ$}
			& \multirow{2}{*}{$\iota$}
			& \multirow{2}{*}{$\cR_X$}
			& \multicolumn{8}{c|}{$\dim$}
			\\
			\cline{5-12}
			
			&
			&
			&
			& $|A|$
			& $|2A|$ 
			& $|3A|$
			& $|4A|$
			& $|A\pm T|$
			& $|2A\pm T|$
			& $|3A\pm T|$
			& $|4A\pm T|$
			\\
			\hline
			
			1
			& $5$
			& $2$
			& $\{2,4,14\}$
			& $0$
			& $0$
			& $0$
			& $1$
			& $-1$
			& $-1$
			& $0$
			& $1$
			\\
			
			\hline
			
			2
			& $5$
			& $3$
			& $\{2,9^2\}$
			& $0$
			& $0$
			& $0$
			& $1$
			& $-1$
			& $0$
			& $1$
			& $2$
			\\		
			
			\hline

			3
			& $5$
			& $2$
			& $\{4^2,12\}$
			& $0$
			& $0$
			& $1$
			& $3$
			& $-1$
			& $0$
			& $1$
			& $3$
			\\		
			
			\hline

			4
			& $5$
			& $2$
			& $\{2^2,4,8\}$
			& $0$
			& $1$
			& $2$
			& $5$
			& $0$
			& $1$
			& $3$
			& $5$
			\\		
			
			\hline

			5
			& $5$
			& $2$
			& $\{2,4^2,6\}$
			& $0$
			& $1$
			& $3$
			& $7$
			& $0$
			& $2$
			& $4$
			& $6$
			\\		
			
			\hline

			6
			& $7$
			& $2$
			& $\{2,6,10\}$
			& $0$
			& $0$
			& $0$
			& $1$
			& $-1$
			& $0$
			& $1$
			& $2$
			\\		
			
			\hline
			
			7
			& $7$
			& $2$
			& $\{2^2,3,4,8\}$
			& $-1$
			& $0$
			& $1$
			& $2$
			& $0$
			& $0$
			& $1$
			& $2$
			\\		
			
			\hline
		\end{longtable}
		\normalsize
	\end{prop}
	
	\begin{proof}
	  Note that we have $b_X=125/37$ and $b_X=5$ in cases \textnumero\,4 and \textnumero\,5 of \cite[Proposition 3.4]{Prokhorov2022} (see also \cite[\textnumero\,41424 and \textnumero\,41431]{BrownKasprzyk2009}), respectively, which are thus ruled out by Theorem~\ref{t.LiuLiu3folds}. The dimensions of $|kA|$ and $|kA\pm T|$ can be derived from the orbifold Riemann--Roch formula (see \cite[Corollary~10.3]{Reid1987} or \cite[(2.3)]{LiuLiu2023}) using the numerical data given in \cite[Proposition 3.4]{Prokhorov2022}. Here we note that $kA-T\sim kA+T$ if $\iota=2$ and $kA-T\sim kA+2 T$ if $\iota=3$.
	\end{proof}

   The following result was pointed out to us by the ‌anonymous referee.  It is obtained by a computer search using an algorithm outlined in \cite[Appendix A]{Prokhorov2024}; see also \cite[Section~3]{Prokhorov2024} for related results.
 
	\begin{lem}
		\label{l.q=3dimensionofA}
		Let $X$ be a terminal $\bQ$-Fano threefold with $q\bQ(X)=3$ such that $\Cl(X)$ is not torsion-free. Let $A$ be any Weil divisor on $X$ such that $-K_X\sim_{\bQ} 3A$. Then $\dim|A|\leq 1$.
	\end{lem}

	\section{Case \texorpdfstring{$\boldsymbol{q\bQ(X)=5}$}{qQ(X)=5}}
	\label{s.qbQ=5}
	
	The goal of this section is to prove the following result.
	
	\begin{thm}
		\label{t.qQ=5}
		Let $X$ be a terminal $\bQ$-Fano threefold such that $q\bQ(X)=5$. Then we have
		\[
		c_1(X)^3<\frac{121}{41} c_2(X)c_1(X).
		\]
	\end{thm}
	
	Assume to the contrary that there exists a terminal $\bQ$-Fano threefold $X$ such that $q\bQ(X)=5$ and $b_X\geq 121/41$. 
	Then Lemma~\ref{c.QFano3foldscandidates} shows that the numerical type of $X$ appears as~\eqref{c.r.5}. In particular, the numerical invariants of $X$ are as follows (\textit{cf.} \cite[\textnumero\,41436]{BrownKasprzyk2009}):
	\[
	\bfB_X=\{(3,1),(7,2),(7,3)\},\quad c_1(X)^3=\frac{500}{21},\quad c_2(X)c_1(X)=\frac{160}{21} \quad \textup{and}\quad b_X=\frac{25}{8}.
	\]
	In this case, the tangent sheaf $\sT_X$ is not semi-stable by the Bogomolov--Gieseker inequality, and it follows from \cite[Theorem 4.2 and Corollary 4.3]{LiuLiu2023} that the maximal destabilizing subsheaf $\sF$ of $\sT_X$ is of rank two with $c_1(\sF)\equiv 4A$, where $A$ is a generator of $\Cl(X)/\sim_{\bQ}$.
	
	\subsection{Foliation}
	
	Let $X$ be a normal variety. A \emph{foliation} on $X$ is a non-zero coherent subsheaf $\sF$ of the tangent sheaf $\sT_X$ such that
	\begin{enumerate}
	\item $\sF$ is saturated in $\sT_X$ (\textit{i.e.} 
          $\sT_X/\sF$ is torsion-free), 
		
		\item $\sF$ is closed under the Lie bracket.
	\end{enumerate}
	The \emph{rank} of $\sF$ is defined to be the generic rank $r$ of $\sF$,  and the \emph{codimension} of $\sF$ is defined as $\dim X -r$.
	
	Given a rank $r$ foliation $\sF$ on a normal variety $X$, the inclusion $\sF\hookrightarrow \sT_X$ induces a non-zero map
	\[
	\eta\colon \Omega^r_X \longrightarrow \wedge^r \sT_X^* \longrightarrow \wedge^r \sF^* \longrightarrow \det(\sF^*).
	\]
	The \emph{singular locus} of $\sF$ is the singular scheme of the map $\eta$; \textit{i.e.} it is the closed subscheme of $X$ whose ideal sheaf is the image of the induced map 
	\[
	(\Omega_X^r\otimes \det(\sF))^{**}\longrightarrow \sO_X.
	\]
	A point $x\in X$ is called a \emph{regular point} of $\sF$ if $x$ is not contained in the singular locus of $\sF$. If $x$ is a regular point of $X$, then $\sF$ is regular at $x$ if and only if $\sT_X/\sF$ is locally free at $x$. In particular, the singular locus of $\sF$ has codimension at least two. Moreover, if both $X$ and $\sF$ are regular at $x$, then by the holomorphic Frobenius theorem, there exists a unique leaf of $\sF$ passing through $x$.
	
	A foliation $\sF$ is called \emph{algebraically integrable} if there exists a dominant rational map $\varphi\colon X\dashrightarrow Y$ to a normal variety $Y$ such that $\sF=\ker(\sT_X\rightarrow \varphi^*\sT_Y)$.
	
	\begin{lem}
		\label{l.singularlocus}
		Let $X$ be a normal projective variety of dimension $n$ which is regular in codimension two. Let $\sF$ be a codimension one algebraically integrable foliation on $X$. If $X$ is $\bQ$-factorial and $\rho(X)=1$, then the singular locus of $\sF$ has codimension two.
	\end{lem}
	
	\begin{proof}
		Assume to the contrary that $\sF$ is regular in codimension two. By assumption, the foliation $\sF$ is induced by a rational map $\varphi\colon X\dashrightarrow Y$ to a curve. Let $F$ and $F'$ be the closures of two general fibres of $\varphi$. Since $X$ is $\bQ$-factorial and $\rho(X)=1$, the divisors $F$ and $F'$ are ample and $\dim(F\cap F')=n-2$. Since both of $X$ and $\sF$ are regular in codimension two, there exists a point $x\in F\cap F'$ such that $X$ and $\sF$ are both regular at $x$. In particular, both of $F$ and $F'$ are the leaves of $\sF$ at $x$, which contradicts the uniqueness of leaves.\looseness=-1
	\end{proof}
	
	\subsection{Proof of Theorem~\ref{t.qQ=5}}
	
	Let $X$ be a terminal $\bQ$-Fano variety of dimension $n$.
	Denote by $r$ the rank of the maximal destabilizing subsheaf $\sF$ of $\sT_X$ and by $k$ the length of the Harder--Narasimhan filtration of $\sT_X$. Note that $\sT_X$ is semi-stable if and only if $(r,k)=(n,1)$. In this case, the Bogomolov--Gieseker inequality and \cite[Theorem 1.2]{GrebKebekusPeternell2021} imply
	\[
	c_1(X)^n <3 c_2(X)c_1(X)^{n-2}.
	\]
	In \cite[Proposition 3.8]{LiuLiu2023}, the authors studied the case $(r,k)=(1,k)$ for $k\geq 2$. Now we study the case $(r,k)=(n-1,2)$ and prove the following result.
	
	\begin{prop}
		\label{p.ineqforn-1}
		Let $X$ be a terminal $\bQ$-Fano variety of dimension $n$. Assume that the maximal destabilizing subsheaf $\sF$ of $\sT_X$ has rank $n-1$. Then we have
		\[
		c_2(X)c_1(X)^{n-2} > \frac{(2(n-1)c_1(X)-nc_1(\sF)) c_1(\sF)}{2(n-1)} c_1(X)^{n-2}.
		\]
	\end{prop}
	
	\begin{proof}
		Since $\sT_X$ is generically ample by \cite[Proposition 3.6]{LiuLiu2023}, it follows from \cite[Lemma 4.10]{CampanaPaun2019} that $\sF$ is actually a foliation. Moreover, as $\sF$ is semi-stable and $\det(\sF)$ is ample, the foliation $\sF$ is algebraically integrable by \cite[Theorem 1.1]{CampanaPaun2019}.\footnote{Though \cite[Lemma 4.10 and Theorem 1.1]{CampanaPaun2019} are only stated for non-singular projective varieties, it is easy to generalise them to $\bQ$-factorial normal projective varieties by taking a resolution.} Consider the following short exact sequence of coherent sheaves: 
		\[
		0\longrightarrow \sF\longrightarrow \sT_X \longrightarrow \sQ \longrightarrow 0.
		\]
		By Lemma~\ref{l.singularlocus}, the singular locus of the foliation $\sF$ has codimension two. In particular, the non--locally free locus of the quotient $\sQ$ is supported on a closed subset of $X$ with codimension two, and therefore $c_2(\sQ)$ is represented by a non-zero effective codimension two cycle. So we get
		\begin{align*}
			c_2(\sT_X) c_1(X)^{n-2} & = (c_2(\sF)+c_2(\sQ) + c_1(\sF)c_1(\sQ))c_1(X)^{n-2} \\
			& >c_2(\sF)c_1(X)^{n-2}+c_1(\sF)c_1(X)^{n-1}-c_1(\sF)^2c_1(X)^{n-2} \\
			& \geq c_1(\sF)c_1(X)^{n-1}-\frac{n}{2n-2} c_1(\sF)^2c_1(X)^{n-2},
		\end{align*}
		where the last inequality follows from the Bogomolov--Gieseker inequality.
	\end{proof}
	
	\begin{proof}[Proof of Theorem~\ref{t.qQ=5}]
		Assume to the contrary that there exists a terminal $\bQ$-Fano threefold $X$ such that $b_X\geq 121/41$. By the discussion at the beginning of this section, the maximal destabilizing subsheaf $\sF$ of $\sT_X$ is of rank two and $c_1(\sF)\equiv 4A$. In particular, as $n=3$ and $c_1(X)\equiv 5A$, it follows from Proposition~\ref{p.ineqforn-1} that 
		\[
		c_2(X)c_1(X) > 8A^2c_1(X) = \frac{8}{25} c_1(X)^3,
		\]
		which contradicts the fact that $b_X=25/8$.
	\end{proof}
	
	\section{Sarkisov link}\label{sec.sarkisov}
	Let $X$ be a terminal $\bQ$-Fano threefold such that $\Cl(X)$ is torsion-free. In particular, $qW(X)=q\bQ(X)$. Set $q\coloneqq qW(X)=q\bQ(X)$. Let $A$ be an ample Weil divisor that generates the group $\Cl(X)\simeq \bZ$.
	So $-K_X\sim qA$. Following \cite{Alexeev1994},
	let $\sM$ be a movable linear system on $X$, and let $c\coloneqq\ct(X,\sM)$ be the canonical threshold of $(X,\sM)$. 
	Then the pair $(X,c\sM)$ is canonical but not terminal. Assume that $-(K_X+c\sM)$ is ample. 
	Let $f\colon \tilde X\to X$ be a $(K+c\sM)$-crepant blowup such that
	\[
	K_{\tilde X}+c\tilde{\sM}=f^*(K_X+c\sM),
	\]
	where $\tilde \sM$ is the strict transform of $\sM$ and $\tilde X$ has only terminal $\bQ$-factorial singularities.
	
	\begin{lem}[\textit{cf.} \protect{\cite[Lemma~4.2]{Prokhorov2010}}]\label{lem.ct}
		Let $P\in X$ be a point of local index $r>1$. Assume that  $\sM\sim -tK_X$ near $P$, where $0<t<r$. Then $c\leq 1/t$.
	\end{lem}
	
	By \cite[Section~4.2]{Alexeev1994} (see also \cite[Section~4.3]{Prokhorov2010}),
	there exists a diagram called a \emph{Sarkisov link $($of type I or type II\,$)$} as follows: 
	\[
	\xymatrix{
		&\tilde X \ar[dl]_{f}\ar@{-->}[r]^{\chi}& \bar X\ar[dr]^-{\hat f}&\\ 
		X&& &\hat X, 
	}
	\]
	where $\tilde X$ and $\bar X$ have only $\bQ$-factorial terminal singularities, $\rho(\tilde X)=\rho(\bar X)=2$, the morphism $f$ is a Mori extremal divisorial contraction, the rational map $\chi$ is a sequence of log flips and $\hat f$ is a Mori extremal contraction (either divisorial or of fibre type).
	
	In what follows, for a divisor $D$ (also a linear system) on $X$, we denote by $\tilde D$ and $\bar D$ the strict transforms of $D$ on $\tilde X$ and $\bar X$, respectively; if $\hat f$ is birational, then we put $\hat D=\hat f_*\bar D$. Conversely, for a divisor $\hat D$ (also a linear system) on $\hat X$, we denote by $\bar D$ and $\tilde D$ the strict transforms of $\hat D$ on $\bar X$ and $\tilde X$, respectively; if $\hat f$ is birational and $\bar F$ is the $\hat f$-exceptional divisor, then we denote by $\tilde F$ the strict transform of $\bar F$ on $\tilde X$ and put $F=f(\tilde F)$.
	
	Let $\tilde E$ be the $f$-exceptional divisor. Set $\sS_k\coloneqq |kA|$ for $k\geq 1$. Write 
	\begin{equation}\label{eq.twoeq}
		\begin{split}
			K_{\tilde X} \sim_{\bQ} f^*K_X+\alpha \tilde E,\quad
			\tilde \sS_k \sim_{\bQ} f^*\sS_k-\beta_k \tilde E,
		\end{split} 
	\end{equation}
	where $\alpha\in \bQ_{>0}$ and $\beta_k\in \bQ_{\geq 0}$ for $k\geq 1$. Since $\Cl(X)$ is torsion-free and $-kK_X\sim q\sS_k$, the relations \eqref{eq.twoeq} provide that for any $k\geq 1$, we have
	\[
	q\beta_k-k\alpha\in \bZ.
	\]
	
	Assume that the morphism $\hat f$ is birational (Sarkisov link of type I). Then $\hat X$ is a terminal $\bQ$-Fano threefold. Set $\hat q\coloneqq q\bQ(\hat X)$ and let $A_{\hat X}$ be an ample Weil divisor on $\hat{X}$ that generates the group $\Cl(\hat X)/\sim_{\bQ}$.
	Denote by $\bar F$ the $\hat f$-exceptional divisor. Write 
	\[
	F\sim dA, \quad \hat E\sim_{\bQ} eA_{\hat X}, \quad \hat \sS_k\sim_{\bQ} s_k A_{\hat X},
	\]
	where $d, e, s_k\in \bZ_{\geq 0}$.
	
	Assume that the morphism $\hat f$ is not birational  (Sarkisov link of type II). Denote by $\bar F$ a general geometric fibre of $\hat f$, which is either a non-singular rational curve or a non-singular del Pezzo surface. The image of the restriction map $\Cl (\bar X) \to \Pic (\bar F)$ is isomorphic to $\bZ$. Let $A_{\bar F}$ be an ample generator of this image. Write
	\[
	-K_{\bar X}|_{\bar F}=-K_{\bar F}\sim \hat qA_{\bar F}, \quad \bar E|_{\bar F}\sim eA_{\bar F}, \quad \bar \sS_k|_{\bar F}\sim s_k A_{\bar F},
	\]
	where $\hat q, e, s_k\in \bZ_{\geq 0}$.
	
	\begin{prop}[\textit{cf.} \protect{\cite[Section~4]{Prokhorov2010}}]
		\label{prop.sarkisov}
		The notation is as above. 
		\begin{enumerate}
			\item\label{i1.p-sarkisov}  The integer $e$ is positive. Moreover, if $\hat f$ is birational, then $d/e$ is the order of the torsion subgroup $\Cl(\hat X)_t$ of $\Cl(\hat X)$.
			\item\label{i2.p-sarkisov} Assume that $\hat f$ is birational. Then $s_k=0$ if and only if $\dim \sS_k=0$ and the unique element of $\bar{\sS}_k$ is $\bar{F}$. Moreover, if $s_k>0$ and $\Cl (\hat X)$ is torsion-free, then $\hat{\sS}_k \sim s_k A_{\hat{X}}$ and thus
			$
			\dim |s_kA_{\hat X}|\geq \dim \hat \sS_k = \dim \sS_k
			$.
                        
			\item\label{i3.p-sarkisov} If\, $\hat q\geq 4$, then $\hat f$ is birational.
			
			\item\label{i4.p-sarkisov} If\, $\hat{f}$ is not birational and $\hat q=3$, then $\bar{F}\cong \bP^2$ and $\hat{X}\cong \bP^1$.
		\end{enumerate}
	\end{prop}
	\begin{proof}
		The first part of~\eqref{i1.p-sarkisov} follows directly from \cite[Claim 4.6]{Prokhorov2010}, and the second part follows from \cite[Lemma 4.13]{Prokhorov2010}. 
		The assertions of~\eqref{i2.p-sarkisov} follows from the definition.
		For the assertions of~\eqref{i3.p-sarkisov} and~\eqref{i4.p-sarkisov}, we assume that $\hat{f}$ is not birational. Since $\bar{F}$ is either a non-singular rational curve or a non-singular del Pezzo surface, it follows that $\hat{q}\leq 3$, where  equality holds only if $\bar{F}\cong \bP^2$.
	\end{proof}
	
	The following very important equality, which indicates the relation between the various invariants introduced above,  will be intensively used in Sections~\ref{s.qbQ=4} and~\ref{s.qQ=8}.
	
	\begin{lem}[\textit{cf.} \protect{\cite[Section~4.8]{Prokhorov2022}}]
		The notation is as above. For any $k\geq 1$, we have
		\begin{equation}\label{eq.key} 
			k\hat q=qs_k+(q\beta_k-k\alpha)e.
		\end{equation}
	\end{lem}
	
	For any point $P\in X$ of local index $r_P$, recall that the local Weil divisor class group $\Cl(X,P)$ is a cyclic group of order $r_P$ generated by $-K_X$, by \cite[Corollary 5.2]{Kawamata1988}. Let $t_k(P)$ be the local index of $\sS_k$ near~$P$. When the context is clear, we shall omit $P$ in the notation and abbreviate $t_k(P)$ to $t_k$. The following fact allows us to determine the possible value of $\beta_k$ once the value of $\alpha$ is known (see \cite{Kawamata1996,Kawakita2005} and \cite[Lemma 2.6]{Prokhorov2022}).
	
	\begin{lem}
		\label{l.localindex-integer}
		If\, $P\coloneqq f(\tilde{E})$ is a closed point of\, $X$, then we have
		\begin{equation}
			\label{e.localindex-integer}
			\beta_k-t_k\alpha\in \bZ.
		\end{equation}
	\end{lem}
	
	\begin{proof}
		This follows from the relations \eqref{eq.twoeq} and the fact that $t_kK_X+\sS_k$ is Cartier near $P$.
	\end{proof}
	
	\section{Case \texorpdfstring{$\boldsymbol{q\bQ(X)=4}$}{qQ(X)=4}}
	\label{s.qbQ=4}
	
	The aim of this section is to prove the following result.
	
	\begin{thm}
		\label{t.qQ=4}
		Let $X$ be a terminal $\bQ$-Fano threefold such that $q\bQ(X)=4$. Then we have
		\[
		c_1(X)^3< \frac{121}{41} c_2(X)c_1(X).
		\]
	\end{thm}
	
	Assume to the contrary that there exists a terminal $\bQ$-Fano threefold $X$ such that $q\bQ(X)=4$ and $b_X\geq 121/41$. Then Lemma~\ref{c.QFano3foldscandidates} shows that the numerical type of $X$ appears as~\eqref{c.r.4}. In particular, the numerical invariants of $X$ are as follows (\textit{cf.} \cite[\textnumero\,41313]{BrownKasprzyk2009}):
	\[
	\bfB_X=\{(7,2),(13,6)\},\quad c_1(X)^3=\frac{1152}{91}, \quad c_2(X)c_1(X)=\frac{384}{91} \quad \textup{and} \quad b_X=3.
	\]
	
	\subsection{Geometry of $\boldsymbol{X}$}
	
	We collect  some basic geometric properties of $X$, which will be used later in the proof of Theorem~\ref{t.qQ=4}.
	
	\begin{lem}
		\label{l.invariantsofX}
		Let $X$ be a terminal $\bQ$-Fano threefold with numerical type~\eqref{c.r.4} in Lemma~\ref{c.QFano3foldscandidates}.
		\begin{enumerate}
			\item\label{item.non-Gorenstein} Every non-Gorenstein point of\, $X$ is a cyclic quotient singularity.
			
			\item\label{item.torsionfree} The Weil divisor class group $\Cl(X)$ is torsion-free.
			
			\item\label{item.dimension} Let $A$ be a $\bQ$-Cartier Weil divisor on $X$ such that $-K_X\sim 4 A$. Then we have
			\renewcommand*{\arraystretch}{1.6}
			\begin{longtable*}{|M{2cm}|M{0.8cm}|M{0.8cm}|M{0.8cm}|M{0.8cm}|M{0.8cm}|M{0.8cm}|}
				\hline
				
				\multirow{2}{*}{$\cR_X$}
				& \multirow{2}{*}{$A^3$}
				& \multicolumn{4}{c|}{$\dim|kA|$}
				\\
				\cline{3-6}
				
				&
				& $|A|$
				& $|2A|$ 
				& $|3A|$
				& $|4A|$ 
				\\
				\hline

				$\{7,13\}$
				& $\frac{18}{91}$
				& $-1$
				& $1$
				& $3$
				& $6$
				\\
				\hline
			\end{longtable*}
			
			\item\label{item.irreandreduced} Every divisor contained in the linear system $|2A|$ or $|3A|$ is reduced and irreducible.
		\end{enumerate}
		
	\end{lem}
	
	\begin{proof}
		The first statement follows from the form of the basket $\bfB_X$ and \cite{Mori1985} (see \cite[Sections~(6.1) and (6.4)]{Reid1987} for more details). The second statement follows from Lemma~\ref{l.Fanothreefoldscandidates} (see also \cite[Proposition 2.9]{Prokhorov2010}). The dimension of $|kA|$ can be derived from the orbifold Riemann--Roch formula or \cite[\textnumero\,41313]{BrownKasprzyk2009}. The last statement follows from the facts that $\dim|A|=-1$ and $\Cl(X)$ is torsion-free.
	\end{proof}

	We will use the Sarkisov link introduced in Section~\ref{sec.sarkisov} to prove Theorem~\ref{t.qQ=4}. This idea was initiated by Yu.~Prokhorov in \cite{Prokhorov2010}. To be more precise, following the notation in Section~\ref{sec.sarkisov}, we consider the Sarkisov link associated to the movable linear system $\sM \coloneqq \sS_3=|3A|$.
	
	\begin{lem}\label{lem.abforq4}
		We have $\beta_3\geq 6\alpha$, and if $\hat f$ is birational, then $d\geq 2$.
	\end{lem}
	
	\begin{proof}
		We apply Lemma~\ref{lem.ct} with $P$ being the point of local index $7$ on $X$, where $c\leq 1/6$ as $\sS_3\sim -6K_X$ near the point $P$. Since $c=\alpha/\beta_3$ by \eqref{eq.twoeq}, the first inequality follows.  For the second statement, note that $\dim|A|=-1$, so we have $d\geq 2$ if $\hat{f}$ is birational. 
	\end{proof}

	The following observation can be easily derived from the definition and Lemma~\ref{l.invariantsofX}. It will play a key role in the proof of Theorem~\ref{t.qQ=4}.
	
	\begin{lem}
		\label{l.hatS2andS3}
		Any element in $\bar{\sS}_2$ $($resp.\ $\bar{\sS}_3)$ is of the form $\bar{\Delta}+a\bar{E}$, where $a$ is a non-negative integer and $\bar{\Delta}$ is the strict transform of an element $\Delta$ in $\sS_2$ $($resp.\ $\sS_3)$. In particular, the divisor $\bar{\Delta}$ is reduced and irreducible. Moreover, if $\hat{f}$ is birational and $\hat{f}_*\bar{\Delta}=0$, then $\bar{\Delta}=\bar{F}$ and $\Delta=F$.
	\end{lem}

    The remainder of this section is dedicated to the proof of Theorem~\ref{t.qQ=4}, which is conducted on a case-by-case basis in accordance with the type of $f(\tilde{E})$. The proof concludes by demonstrating that no instance of $f(\tilde{E})$ can be realised.
	
	\subsection{The image $\boldsymbol{f(\tilde{E})}$ is a curve or a Gorenstein point}
	
	If $f(\tilde{E})$ is either a curve or a Gorenstein point of $X$, then $\alpha$ and $\beta_3$ are integers, and then it follows from \eqref{eq.key} and Lemma~\ref{lem.abforq4}  that
	\begin{equation}
     \label{eq.q5Gorenstein}
		3\hat q=4s_3+(4\beta_3-3\alpha)e\geq 4s_3+21\alpha e.
	\end{equation}
	In particular, we get $\hat q\geq 7\alpha e\geq 7$, and therefore it follows from Proposition~\ref{prop.sarkisov} that the morphism $\hat f$ is birational. Moreover, as $\dim\hat{\sS}_3=3$, we get $s_3>0$ by Proposition~\ref{prop.sarkisov} and hence $\hat q\geq 9$. 
	By Proposition~\ref{p.torsion}, the group $\Cl(\hat X)$ is torsion-free. Thus $e=d=2$ and $\alpha=1$ by Proposition~\ref{prop.sarkisov} as $\hat{q}\leq 19$. So $\hat{q}=17$ or $\hat{q}=19$ by Lemma~\ref{l.indexQFano}. However, there are no integral solutions for \eqref{eq.q5Gorenstein} in these two cases, so we have a contradiction.

	\subsection{The image $\boldsymbol{f(\tilde{E})}$ is a point of local index 13}
	
	If $f(\tilde{E})$ is a point of local index $13$, then $\alpha=1/13$ by \cite{Kawamata1996}.
	Moreover, as $A\sim -10K_X$ near the point $f(\tilde{E})$, we obtain $t_1=10$, $t_2=7$ and $t_3=4$. Then it follows from \eqref{e.localindex-integer} and Lemma~\ref{lem.abforq4} that $\beta_2=7/13+m_2$ for some $m_2\in \bZ_{\geq 0}$ and $\beta_3=4/13+m_3$ for some $m_3\in \bZ_{>0}$. Applying \eqref{eq.key} to $k=3$ yields 
	\begin{equation}
		\label{e.index13keyeq}
		3\hat q=4s_3+(4m_3+1)e=4(s_3+m_3e)+e.
	\end{equation}
	If $\hat{q}\geq 4$, by Proposition~\ref{prop.sarkisov}, the morphism $\hat{f}$ is birational and $s_3>0$ as $\dim\hat{\sS}_3 = 3$. Then we derive from \eqref{e.index13keyeq} that $\hat{q}\not= 4$ or $5$; the possibilities for $\hat{q}=6$ and $7$ are listed  below. If $\hat q\geq 8$, then $\Cl(\hat X)$ is torsion-free by Proposition~\ref{p.torsion}, so $e=d\geq 2$ and $\dim |s_3A_{\hat X}|\geq \dim \hat \sS_3 = 3$ by Proposition~\ref{prop.sarkisov}. Combining these facts with \eqref{e.index13keyeq} and Table~\ref{table.qgeq6} yields the corresponding possible values of $s_3$ for each $\hat q$. In summary, the possibilities for $(\hat q, e, s_3)$ are as follows:
	\begin{itemize}
		\item $\hat q=3$, $e=1$ and $0\leq s_3\leq 1$;
		\item $\hat q=6$, $e=2$ and $s_3=2$;
		\item $\hat q=7$, $e=1$ and $1\leq s_3\leq 4$;
		\item $\hat q=13$, $e=3$ and $s_3=6$;
		\item $\hat q=17$, $e=3$ and $s_3=9$.
	\end{itemize}
	This case will be treated on a case-by-case basis according to the value of $\hat{q}$.
	
	\subsubsection{The case $\boldsymbol{\hat{q}=3}$ and $\boldsymbol{e=1}$}
	
	In this case, we have $0\leq s_3\leq 1$ and $s_2=1-m_2\leq 1$ by \eqref{eq.key}. Now we divide the proof into two subcases according to the type of $\hat{f}$.
	
	
	\paragraph{\textit{The subcase where $\hat f$ is birational}}
	
	As $d=|\Cl(\hat{X})_t|\geq 2$ by Proposition~\ref{prop.sarkisov}, the group $\Cl(X)$ is not torsion-free. Thus we get a contradiction by Lemma~\ref{l.q=3dimensionofA} as $s_3=1$ and $\dim\hat{\sS}_3=3$.
	
	
	\paragraph{\textit{The subcase where $\hat{f}$ is not birational}}
	\label{para.hatfnotbirationl}
	
	Then the general fibre $\bar{F}$ of $\hat{f}$ is isomorphic to $\bP^2$ and $\hat{X}\cong \bP^1$ by Proposition~\ref{prop.sarkisov}. Moreover, if $\bar{\sS}_3$ is $\hat f$-vertical, then $\bar{\sS}_3$ consists of fibres of $\hat f$, which contradicts $\dim\bar{\sS}_3=3$ and Lemma~\ref{l.hatS2andS3}. So we obtain $s_3=1$. Consider the restriction map
	\[
	\bar{\sS}_3 \longrightarrow \bar{\sS}_3|_{\bar{F}} \subset |\sO_{\bP^2}(1)|.
	\]
	As $\dim|\sO_{\bP^2}(1)|=2$, $\dim \bar{\sS}_3=3$ and $e=s_3=1$, we get $\bar{F}+\bar{E} \in \bar{\sS}_3$ by Lemma~\ref{l.hatS2andS3}, so $\hat{f}^*|\sO_{\bP^1}(1)|+\bar{E}\sim \bar{\sS}_3$. Pushing it forward to $X$ yields $
	f_*\chi^*\hat{f}^*|\sO_{\bP^1}(1)| \subset \sS_3$, which implies $\hat{f}^*|\sO_{\bP^1}(1)| + \bar{E} \subset \bar{\sS}_3$. Then it follows from Lemma~\ref{l.hatS2andS3} that the (cycle-theoretic) fibres of $\hat{f}$ are reduced and irreducible. 
	
	We claim that $s_2=1$. Assume to the contrary that $s_2=0$. Since $\dim\bar{\sS}_2=1$ and its general elements are reduced and irreducible, we have $\bar{F}\in \bar{\sS}_2$. This implies that $F\in \sS_2\cap \sS_3$, which is absurd. So $s_2=1$ and the group $\Cl(\bar{X})$ is generated by $\bar{S}_2$ and $\bar{F}$, where $\bar{S}_2$ is a general element in $\bar{\sS}_2$. In particular, as $s_3=1$, there exists an integer $c$ such that $\bar{\sS}_3\sim \bar{S}_2+c\bar{F}$. Then pushing it forward to $X$ shows $A\sim \sS_3 - \sS_2 \sim cF \sim 3c A$, which is impossible.
	
	
	\subsubsection{The case $\boldsymbol{\hat{q}=6}$, $\boldsymbol{e=2}$ and $\boldsymbol{s_3=2}$}

	Then $qW(\hat{X})=q\bQ(\hat{X})=6$ by Lemma~\ref{l.indexQFano}, and the numerical type of $\hat{X}$ appears as \textnumero\,1 in Table~\ref{table.qgeq6}, so $\Cl(\hat X)$ is torsion-free by Proposition~\ref{p.torsion}. In particular, we have $\dim \sS_3=\dim \hat \sS_3=\dim |2A_{\hat X}|=3$. So $\hat{\sS}_3=|2A_{\hat{X}}|$ and thus $2G\in \hat{\sS}_3$ for every $G\in |A_{\hat{X}}|$. This contradicts Lemma~\ref{l.hatS2andS3} as $\dim|A_{\hat{X}}|=1$ and $e=s_3=2$.    
	
	\subsubsection{The case $\boldsymbol{\hat{q}=7}$, $\boldsymbol{e=1}$ and $\boldsymbol{s_3\leq 4}$}\label{ss.hatq=7ande=1}
	
	Since $d/e=d\geq 2$ is the order of the torsion subgroup of $\Cl(\hat X)$, the numerical type of $\hat{X}$ appears as either \textnumero\,6 or \textnumero\,7 in Table~\ref{table.qgeq5} by Proposition~\ref{p.torsion}. Then one obtains a contradiction from Table~\ref{table.qgeq5} as $\dim\hat{\sS}_3=3$ and $1\leq s_3\leq 4$.
	
	
	\subsubsection{The case $\boldsymbol{\hat{q}=13}$, $\boldsymbol{e=3}$ and $\boldsymbol{s_3=6}$}
	
	In this case, the numerical type of $\hat{X}$ appears as \textnumero\,28 in Table~\ref{table.qgeq6}, and $\Cl(\hat X)$ is torsion-free by Proposition~\ref{p.torsion}. As $e=3$ and $\dim|A_{\hat{X}}|=0$, the unique element $D\in |A_{\hat{X}}|$ is a prime divisor which is different from $\hat{E}$. Moreover, note that we have
	\[
	6=\dim |5A_{\hat{X}}| + \dim\hat{\sS}_3 \geq \dim|s_3A_{\hat{X}}|=\dim|6A_{\hat{X}}|=4.
	\] 
	In particular, since both $D+|5A_{\hat{X}}|$ and $\hat{\sS}_3$ are sublinear systems of $|6A_{\hat{X}}|$, there must exist an element in $\hat{\sS}_3$ of the form $D+G$ for some $G\in |5A_{\hat{X}}|$. This contradicts Lemma~\ref{l.hatS2andS3} as $e=3$.
	
	
	\subsubsection{The case $\boldsymbol{\hat{q}=17}$, $\boldsymbol{e=3}$ and $\boldsymbol{s_3=9}$}

	The numerical type of $\hat{X}$ appears as \textnumero\,30 in Table~\ref{table.qgeq6}, and $\Cl(\hat X)$ is torsion-free by Proposition~\ref{p.torsion}. As $e=3$ and $\dim|A_{\hat{X}}|=-1$, the unique element $D\in |2A_{\hat{X}}|$ is a prime divisor which is different from $\hat{E}$. As $\dim\hat{\sS}_3=\dim|9A_{\hat{X}}|=3$, we get $\hat{\sS}_3=|9A_{\hat{X}}|$. So the linear system $\hat{\sS}_3$ contains the divisor $D+G$ for any $G\in |7A_{\hat{X}}|$. This contradicts Lemma~\ref{l.hatS2andS3} as $e=3$ and $\dim|7A_{\hat{X}}|=2$.
	
	\subsection{The image $\boldsymbol{f(\tilde{E})}$ is a point of local index 7}
	
	If $f(\tilde{E})$ is a point of local index $7$, then $\alpha=1/7$ by \cite{Kawamata1996}. Moreover, as $A\sim -2K_X$ near the point $f(\tilde{E})$, we obtain $t_1=2$, $t_2=4$ and $t_3=6$. Then it follows from \eqref{e.localindex-integer} that $\beta_2=4/7+m_2$ for some $m_2\in \bZ_{\geq 0}$ and $\beta_3=6/7+m_3$ for some $m_3\in \bZ_{\geq 0}$. Applying \eqref{eq.key} to $k=3$ yields 
	\[
	3\hat q=4s_3+(4m_3+3)e=4(s_3+m_3e)+3e.
	\]
	As $\hat{q}\leq 19$ by Lemma~\ref{l.indexQFano}, we obtain $\hat{q}-e=4l$ for some integer $0\leq l\leq 4$. Moreover, if $\hat q\geq 8$, then $\Cl(\hat X)$ is torsion-free by Proposition~\ref{p.torsion}, so $e=d\geq 2$ and $\dim|s_3 A_{\hat{X}}| \geq \dim \hat{\sS}_3=3$ by Proposition~\ref{prop.sarkisov}. Thus it follows from Table~\ref{table.qgeq6} that the possibilities for  $(\hat q, e, s_3)$ are the following:
	\begin{itemize}
		\item $\hat q=e$, $s_3=0$;
		\item $\hat q=4+e$, $1\leq e\leq 4$ and $1\leq s_3\leq 3$;
		\item $\hat q=11$, $e=3$ and $s_3=6$;
        \item $\hat q=13$, $e=5$ and $s_3=6$;
		\item $\hat q=17$, $e=5$ and $s_3=9$;
		\item $\hat q=19$, $e=3$ and $s_3=12$.
	\end{itemize}
	In the following, we consider each case in accordance with the value of $\hat{q}$.
	
	\subsubsection{The case $\boldsymbol{\hat{q}=e}$ and $\boldsymbol{s_3=0}$} 
	
	As $\dim \sS_3=3$ and  $s_3=0$, it follows from Proposition~\ref{prop.sarkisov} that $\hat{f}$ is not birational and hence $\hat q=e\leq 3$. Since $\Cl(X)$ is torsion-free, the group $\Cl(\hat{X})$ is so by \cite[Section~2.3]{Prokhorov2013};   we  denote by $A_{\hat{X}}$ its ample generator. Since $\bar\sS_3$ is $\hat{f}$-vertical and $\dim\bar \sS_3=3$, it follows that a general member $\bar{D}$ in $\bar \sS_3$ is the pull-back of a divisor on $\hat{X}$ as $\bar{D}$ is reduced and irreducible by Lemma~\ref{l.hatS2andS3}. Thus we have $\bar{\sS}_3=\hat{f}^*|cA_{\hat{X}}|$ for some $c\in\bZ_{>0}$. Then it follows from Lemma~\ref{l.hatS2andS3} that the elements in $|cA_{\hat{X}}|$ are reduced and irreducible. In particular, as $\dim|A_{\hat{X}}|\geq 0$, we obtain $c=1$ and then $\hat{X}\cong \bP^2$ by \cite[Section~2.3]{Prokhorov2013} as $\dim|\bar\sS_3|=3$. Note that $s_2=0$ by \eqref{eq.key}. In particular, since $\dim\bar{\sS}_2=1$ and the elements in $\bar{\sS}_2$ are reduced and irreducible, applying the same argument as above yields $\bar{\sS}_2=\hat{f}^*|A_{\hat{X}}|$. This is impossible as $\dim|A_{\hat{X}}|=3$.
	
	
	\subsubsection{The case $\boldsymbol{\hat{q}=4+e}$ and $\boldsymbol{1\leq s_3\leq 3}$}
	
	In this case, Proposition~\ref{prop.sarkisov} implies that the morphism $\hat{f}$ is birational and thus $s_3\geq 1$ as $\dim\sS_3=3$. If $\hat{q}\geq 8$, then $\Cl(\hat{X})$ is torsion-free by Proposition~\ref{p.torsion}, so $\dim|s_3A_{\hat{X}}| \geq \dim\hat{\sS}_3 = 3$. Then it follows from Table~\ref{table.qgeq6} that $\hat{q}\leq 8$ and hence $e\leq 4$. We divide the proof into four subcases according to the value of $e$.
	
	
	\paragraph{\textit{The subcase $e=4$}} 
	
	Then $\hat{q}=8$, $s_3=3$, and the numerical type of $\hat{X}$ appears as \textnumero\,19 in Table~\ref{table.qgeq6}. In particular, we have $\dim|3A_{\hat{X}}|=\dim\hat{\sS}_3=3$ and thus $|3A_{\hat{X}}|=\hat{\sS}_3$. As $e=4$ and $\dim|A_{\hat{X}}|=0$, the unique element $D\in |A_{\hat{X}}|$ is a prime divisor which is different from $\hat{E}$. So we have $3D\in\hat{\sS}_3$, which contradicts Lemma~\ref{l.hatS2andS3} as $e=4$.
	
	
	\paragraph{\textit{The subcase $e=3$}}
	
	Then $\hat{q}=7$ and $s_3=3$. In particular, as $\dim\hat{\sS}_3=3$, it follows from Table~\ref{table.qgeq5} that $\Cl(\hat{X})$ is torsion-free and thus $\dim|3A_{\hat{X}}|\geq \dim\hat{\sS}_3=3$. So the numerical type of $\hat{X}$ appears as one of \textnumero\,7, \textnumero\,8 and \textnumero\,9 in Table~\ref{table.qgeq6}. In all these three cases, we have $\dim |A_{\hat{X}}|\geq 0$. In particular, as $e=3$, there exists a prime divisor $D\in |A_{\hat{X}}|$ which is different from $\hat{E}$. Moreover, in all these three cases we have $\dim|2A_{\hat{X}}| + \dim\hat{\sS}_3 \geq \dim|3A_{\hat{X}}|$.	So there exists an element in $\hat{\sS}_3$ of the form $D+G$ for some $G\in |2A_{\hat{X}}|$, which contradicts Lemma~\ref{l.hatS2andS3} as $e=3$.

	
	\paragraph{\textit{The subcase $e=2$}}
	
	Then $\hat{q}=6$ and $s_3=1$ or $3$. By Proposition~\ref{p.torsion}, the group $\Cl(\hat{X})$ is torsion-free and thus $\dim|s_3A_{\hat{X}}|\geq \dim\hat{\sS}_3=3$. It follows that $s_3=3$ and the numerical type of $\hat{X}$ appears as \textnumero\,1 in Table~\ref{table.qgeq6}. On the other hand, we have $s_2=2$ by \eqref{eq.key}. As $\dim|A_{\hat{X}}|=1$, the image of the natural map 
    \[
    |A_{\hat{X}}|\times |A_{\hat{X}}| \longrightarrow |2A_{\hat{X}}|,\quad (D,D')\longmapsto D+D'
    \]
    is a two-dimensional sublinear system of $|2A_{\hat{X}}|$. In particular, since $\dim \hat{\sS}_2=1$ and $\dim |2A_{\hat{X}}|=3$, there must exist an element in $\hat{\sS}_2$ of the form $D+D'$ with $D$, $D'\in |A_{\hat{X}}|$. This contradicts Lemma~\ref{l.hatS2andS3} as $e=2$.
	
	
	\paragraph{\textit{The subcase $e=1$}}
	
	Then $\hat{q}=5$ and $1\leq s_2\leq 2$ by \eqref{eq.key}, so $qW(\hat{X})=q\bQ(\hat{X})=5$ by Lemma~\ref{l.indexQFano}. Let $A_{\hat{X}}$ be a Weil divisor on $\hat{X}$ such that $-K_{\hat{X}}\sim 5 A_{\hat{X}}$. As $d\geq 2$, the group $\Cl(\hat{X})$ is not torsion-free by Proposition~\ref{prop.sarkisov}. In particular, as $\dim \hat{\sS}_2=1$ and $1\leq s_2\leq 2$, it follows from Table~\ref{table.qgeq5} that $s_2=2$, $d=2$ and the numerical type of $\hat{X}$ appears as one of \textnumero\,4 and \textnumero\,5 in Table~\ref{table.qgeq5}. On the other hand, since $\Cl(X)$ is torsion-free, we have $F\in \sS_2$. So there exists a non-negative integer $a$ such that $\bar{F}+a\bar{E}\in \bar{\sS}_2$. Pushing it forward to $\hat{X}$ shows that $a=s_2=2$ as $e=1$, \textit{i.e.}, $2\hat{E}\in \hat{\sS}_2$. This yields $\hat{\sS}_2\sim 2\hat{E}\sim 2A_{\hat{X}}$ as $d=|\Cl(\hat{X})_t|=2$. In particular, as $\dim\hat{\sS}_2=\dim|2A_{\hat{X}}|=1$, we must have $\hat{\sS}_2=|2A_{\hat{X}}|$. On the other hand, by Table~\ref{table.qgeq5}, we also have $\dim|A_{\hat{X}}|=\dim|A_{\hat{X}}+T|=0$.	Let $D$ and $D'$ be the unique elements in $|A_{\hat{X}}|$ and $|A_{\hat{X}}+T|$, respectively. Then $D$ and $D'$ are distinct prime divisors. However, since $2D$ and $2D'$ are contained in $\hat{\sS}_2$, it follows from Lemma~\ref{l.hatS2andS3} that we must have $D=D'=\hat{E}$, which is impossible.
	
	\subsubsection{The case $\boldsymbol{\hat{q}=11}$, $\boldsymbol{e=3}$ and $\boldsymbol{s_3=6}$}
	
	The numerical type of $\hat{X}$ appears as one of \textnumero 25 and \textnumero 26 in Table~\ref{table.qgeq6}. As $e=3$, the unique element $D\in |A_{\hat{X}}|$ is a prime divisor which is different from $\hat{E}$. Moreover, note that $\dim|5A_{\hat{X}}| + \dim\hat{\sS}_3 > \dim|6 A_{\hat{X}}|$, so there exists an element in $\hat{\sS}_3$ of the form $D+G$ for some $G\in |5A_{\hat{X}}|$, which contradicts Lemma~\ref{l.hatS2andS3} as $e=3$.

    \subsubsection{The case $\boldsymbol{\hat{q}=13}$, $\boldsymbol{e=5}$ and $\boldsymbol{s_3=6}$}

    The numerical type of $\hat{X}$ appears as \textnumero\,28 in Table~\ref{table.qgeq6}. Then applying \eqref{eq.key} to $k=2$ yields $s_2=4$. Moreover, as $e=5$, the unique element $D\in |A_{\hat{X}}|$ is a prime divisor which is different from $\hat{E}$. Since $2=\dim|3A_{\hat{X}}| + \dim\hat{\sS}_2 = \dim|4A_{\hat{X}}|$, there exists an element in $\hat{\sS}_2$ of the form $D+G$ for some $G\in |3A_{\hat{X}}|$, which contradicts Lemma~\ref{l.hatS2andS3} as $e=5$.
	
	\subsubsection{The case $\boldsymbol{\hat{q}=17}$, $\boldsymbol{e=5}$ and $\boldsymbol{s_3 = 9}$}
	
	The numerical type of $\hat{X}$ appears as \textnumero\,30 in Table~\ref{table.qgeq6}, and $\Cl(X)$ is torsion-free by Proposition~\ref{p.torsion}. Thus $|9A_{\hat{X}}|=\hat{\sS}_3$ as $\dim|9A_{\hat{X}}|=\dim\hat{\sS}_3=3$. As $e=5$ and $\dim|A_{\hat{X}}|=-1$, the unique element $D\in |2A_{\hat{X}}|$ is a prime divisor which is different from $\hat{E}$. So $\hat{\sS}_3=|9A_{\hat{X}}|$ contains the divisor of the form $D+G$ for any $G\in |7A_{\hat{X}}|$. This contradicts Lemma~\ref{l.hatS2andS3} as $e=5$ and $\dim|7A_{\hat{X}}|=2$.
	
	
	\subsubsection{The case $\boldsymbol{\hat{q}=19}$, $\boldsymbol{e=3}$ and $\boldsymbol{s_3=12}$} 
	
	Then applying \eqref{eq.key} to $k=2$ yields $s_2=8$ or $s_2\leq 5$. Since $\Cl(\hat{X})$ is torsion-free by Proposition~\ref{p.torsion}, we have $\dim|s_2 A_{\hat{X}}| \geq \dim\hat{\sS}_2 =1$ by Proposition~\ref{prop.sarkisov}.
	This implies that $s_2=8$ from Table~\ref{table.qgeq6}. In particular, we have $\dim|8A_{\hat{X}}|=\dim\hat{\sS}_2=1$ and thus $|8 A_{\hat{X}}|=\hat{\sS}_2$. As $e=3$ and $\dim|kA_{\hat{X}}|=-1$ for $k\leq 2$, the unique element $D\in |4A_{\hat{X}}|$ is a prime divisor which is different from $\hat{E}$. So  $2D\in \hat{\sS}_2$, which contradicts Lemma~\ref{l.hatS2andS3} as $e=3$.

	\section{Case \texorpdfstring{$\boldsymbol{q\bQ(X)=8}$}{qQ(X)=8}}
	\label{s.qQ=8}
	
	This section is devoted to proving the following result.
	
	\begin{thm}
		\label{t.qQ=8}
		Case \textnumero\,22 in Table~\ref{table.qgeq6} does not occur for terminal $\bQ$-Fano threefolds.
	\end{thm}
	
	Assume to the contrary that there exists a terminal $\bQ$-Fano threefold $X$ whose numerical type appears as \textnumero\,22 in Table~\ref{table.qgeq6}. Then the numerical invariants of $X$ are as follows (\textit{cf.} \cite[\textnumero\,41495]{BrownKasprzyk2009}):
	\[
	\bfB_X=\{(3,1),(5,2),(11,4)\},\quad c_1(X)^3=\frac{2048}{165},\quad c_2(X)c_1(X)=\frac{928}{165}, \quad b_X=\frac{64}{29}.
	\]
	
	\subsection{Geometry of $\boldsymbol{X}$}
	
	We collect in the following some geometric properties of $X$ which will be used later in the proof of Theorem~\ref{t.qQ=8}.
	
	\begin{lem}
		\label{l.geometryqQ=8}
		Let $X$ be a terminal $\bQ$-Fano threefold with numerical type \textnumero\,22 in Table~\ref{table.qgeq6}.
		\begin{enumerate}
			\item Every non-Gorenstein point of\, $X$ is a cyclic quotient singularity.
			
			\item The Weil divisor class group $\Cl(X)$ is torsion-free.
			
			\item Let $A$ be a $\bQ$-Cartier Weil divisor on $X$ such that $-K_X\sim 8 A$. Then we have
			
			\renewcommand*{\arraystretch}{1.6}
			\begin{longtable*}{|M{2cm}|M{0.8cm}|M{0.8cm}|M{0.8cm}|M{0.8cm}|M{0.8cm}|M{0.8cm}|M{0.8cm}|M{0.8cm}|}
				\hline
				
				\multirow{2}{*}{$\cR_X$}
				& \multirow{2}{*}{$A^3$}
				& \multicolumn{6}{c|}{$\dim|kA|$}
				\\
				\cline{3-8}
				
				&
				& $|A|$
				& $|2A|$ 
				& $|3A|$
				& $|4A|$ 
				& $|5A|$
				& $|6A|$
				\\
				\hline

				$\{3,5,11\}$
				& $\frac{4}{165}$
				& $-1$
				& $0$
				& $0$
				& $1$
				& $2$
				& $3$
				\\
				\hline
			\end{longtable*}
			
			\item\label{i.geometryqQ=8} Let $S_2$ be the unique prime divisor in $|2A|$. Then an element $\Delta\in |4A|$ is either equal to $2S_2$ or a prime divisor. 
		\end{enumerate}
	\end{lem}
	
	\begin{proof}
		The first statement follows from the form of $\bfB_X$ and \cite{Mori1985}. The second and third statements follow directly from Proposition~\ref{p.torsion} and Table~\ref{table.qgeq6}. The last statement follows from the facts that $\dim|A|=-1$, $\dim|2A|=0$ and $\Cl(X)$ is torsion-free. 
	\end{proof}
	
	We will always denote by $S_i$ the unique prime divisor in $|iA|$ for $i=2$ and $3$. Following the notation in Section~\ref{sec.sarkisov}, we consider the Sarkisov link associated to the movable linear system $\sM \coloneqq  \sS_4 = |4A|$; the arguments are very similar to the ones used in Section~\ref{s.qbQ=4}.

	\begin{lem}
		\label{l.cqQ=8}
		We have $\beta_4\geq 6\alpha$, and if $\hat f$ is birational, then $d\geq 2$.
	\end{lem}
	
	\begin{proof}
		We apply Lemma~\ref{lem.ct} with $P$ being the point of local index $11$ on $X$, where $c\leq 1/6$ as $\sM=\sS_4\sim -6K_X$ near the point $P$. Then we obtain the first inequality from \eqref{eq.twoeq}. For the second statement, note that $\dim|A|=-1$, so $d\geq 2$ if $\hat{f}$ is birational. 
	\end{proof}
	
	The following simple but useful observation follows from item~\eqref{i.geometryqQ=8} of Lemma~\ref{l.geometryqQ=8}.
	
	\begin{lem}
		\label{l.elementsinMqQ=8}
		Any element in $\bar{\sS}_4$ is of the form $\bar{\Delta}+a\bar{E}$, where $a$ is a non-negative integer and $\bar{\Delta}$ is the strict transform of an element $\Delta$ in $\sS_4$. In particular, if $\bar{\Delta}\not=2\bar{S}_2$, then $\bar{\Delta}$ is a prime divisor, where $\bar{S}_2$ is the strict transform of $S_2$.
	\end{lem}

    As with the previous section, the remainder of this section is dedicated to the proof of Theorem~\ref{t.qQ=8} on a case-by-case basis according to the type of $f(\tilde{E})$.
	
	\subsection{The image $\boldsymbol{f(\tilde{E})}$ is a curve or a Gorenstein point} 
	
	In this case, both $\alpha$ and $\beta_4$ are integers, and it follows from \eqref{eq.key} and Lemma~\ref{l.cqQ=8} that 
	\begin{equation}
		\label{e.qQ=8-nonGorenstein}
		4\hat{q} = 8s_4 + (8\beta_4-4\alpha)e \geq 8s_4 + 44\alpha e \geq 44.
	\end{equation}
	So we have $\hat{q}\geq 11$. In particular, the morphism $\hat{f}$ is birational by Proposition~\ref{prop.sarkisov}, and $\Cl(\hat{X})$ is torsion-free by Proposition~\ref{p.torsion}. So $e=d\geq 2$ by Proposition~\ref{prop.sarkisov}, and \eqref{e.qQ=8-nonGorenstein} implies $\hat{q}\geq 22$, which contradicts Lemma~\ref{l.indexQFano}.
	
	\subsection{The image $\boldsymbol{f(\tilde{E})}$ is a point of local index 3} 
	
	In this case,  we have $\alpha=1/3$ by \cite{Kawamata1996}. Moreover, as $A\sim -2K_X$ near the point $f(\tilde{E})$, we obtain $t_4=2$. In particular, by \eqref{e.localindex-integer} and Lemma~\ref{l.cqQ=8}, there exists an integer $m_4\geq 2$ such that $\beta_4=2/3+m_4$. Then applying \eqref{eq.key} to $k=4$ yields
	\begin{equation}
     \label{e.qQ8-index3}
	    \hat{q}=2s_4+(2m_4+1)e.
	\end{equation}
	As $e\geq 1$ and $m_4\geq 2$, we have $\hat{q}\geq 5$. In particular, the morphism $\hat{f}$ is birational, and thus $s_4\geq 1$ by Proposition~\ref{prop.sarkisov}. Therefore, we have $\hat{q}\geq 7$. 
	
	If $e=1$, then the group $\Cl(\hat{X})$ is not torsion-free by Proposition~\ref{prop.sarkisov} as $d\geq 2$, and then it follows from Proposition~\ref{p.torsion} that $\hat{q}=7$. Then $s_4=1$ by \eqref{e.qQ8-index3}, which contradicts Table~\ref{table.qgeq5} as $\dim\hat{\sS}_4=1$.
	
	If $e\geq 2$, then $\hat{q}\geq 12$ and the group $\Cl(\hat{X})$ is torsion-free by Proposition~\ref{p.torsion}. Thus we have the inequality $\dim|s_4 A_{\hat{X}}|\geq \dim\hat{\sS}_4=1$.	Then one can easily derive from Table~\ref{table.qgeq6} that there is no solution for \eqref{e.qQ8-index3}.

	\subsection{The image $\boldsymbol{f(\tilde{E})}$ is a point of local index 5}
	
	In this case,  we have $\alpha=1/5$ by \cite{Kawamata1996}. Moreover, as $A\sim -2K_X$ near the point $f(\tilde{E})$, we obtain $t_2=4$, $t_3=1$ and $t_4=3$. Then it follows from \eqref{e.localindex-integer} and Lemma~\ref{l.cqQ=8} that $\beta_2=4/5+m_2$ for some $m_2\in \bZ_{\geq 0}$, $\beta_3=1/5+m_3$ for some $m_3\in \bZ_{\geq 0}$ and $\beta_4=3/5+m_4$ for some $m_4\in \bZ_{>0}$. Now applying \eqref{eq.key} to $k=4$ yields\looseness=1
	\begin{equation}
 \label{eq.f(E)index5}
	    \hat{q}=2s_4 + (2m_4+1)e.
	\end{equation}
	As $e\geq 1$ and $m_4\geq 1$, we have $\hat{q}\geq 3$. On the other hand, if $\hat{q}\geq 4$, then $\hat{f}$ is birational and $s_4\geq 1$ by Proposition~\ref{prop.sarkisov}. So $\hat{q}\not=4$ or $6$. If $\hat{q}\geq 8$, then $\Cl(\hat{X})$ is torsion-free by Proposition~\ref{p.torsion}, so we have $\dim|s_4 A_{\hat{X}}| \geq \dim\hat{\sS}_4=1$ and $e=d\geq 2$ by Proposition~\ref{prop.sarkisov}. Then one can derive from Table~\ref{table.qgeq6} that there is no solution for \eqref{eq.f(E)index5} in the case $\hat{q}\geq 8$. In conclusion, the possibilities for $(\hat{q},e,s_4)$ are the following:
	\begin{itemize}
		\item $\hat{q}=3$, $e=1$ and $s_4=0$;
		
		\item $\hat{q}=5$, $e=1$ and $s_4=1$;
		
		\item $\hat{q}=7$, $e=1$ and $1\leq s_4\leq 2$.
	\end{itemize}
	Moreover, the last two cases can be easily excluded by applying \eqref{eq.key} to $k=3$. Thus it remains to consider the first case. In this case, by Proposition~\ref{prop.sarkisov}, the morphism $\hat{f}$ is not birational, and we have $\hat{X}\cong \bP^1$. In particular, $\bar{\sS}_4$ is $\hat f$-vertical as $s_4=0$ and hence $\bar{\sS}_4=\hat{f}^*|\sO_{\bP^1}(1)|$ as $\dim \bar{\sS}_4=1$. Thanks to Lemma~\ref{l.elementsinMqQ=8}, the divisor $2\bar{S}_2\in \bar{\sS}_4$ is the only (cycle-theoretic) fibre of $\hat{f}$ which is not a prime divisor on~$\bar{X}$. Then applying \eqref{eq.key} to $k=3$ yields $s_3+m_3=1$. If $s_3=0$, then $\bar{S}_3$ is $\hat{f}$-vertical. In particular, there exists a positive integer $a$ such that $a\bar{S}_3\in \bar{\sS}_4$. Pushing it forward to $X$ yields $aS_3\in \sS_4$, which gives a contradiction. So we have $s_3=e=1$ and $m_3=0$. Notice that $\Cl(\bar X)$ is generated by $\bar S_2$ and $\bar E$, so there exists an integer $c$ such that $\bar{S}_3 - \bar{E} \sim c\bar{S}_2$, which implies that $S_3 \sim cS_2$ by pushing forward to $X$. This is absurd.
	
	\subsection{The image $\boldsymbol{f(\tilde{E})}$ is a point of local index 11}
	
	Then we have $\alpha=1/11$ by \cite{Kawamata1996}. Moreover, as $A\sim -7K_X$ near the point $f(\tilde{E})$, we obtain $t_2=3$, $t_3=10$, $t_4=6$, $t_5=2$ and $t_6=9$. In particular, by \eqref{e.localindex-integer}, we have $\beta_k=t_k/11+m_k$ for some $m_k\in \bZ_{\geq 0}$ ($2\leq k\leq 6$). Then applying \eqref{eq.key} to $k=3$ yields
	\begin{equation}
		\label{e.k3index11qQ8}
		3\hat{q} = 8s_3 + (8m_3 + 7)e.
	\end{equation}
	Then one can derive that $3\hat{q}\geq 14$ and hence $\hat{q}\geq 5$. In particular, the morphism $\hat{f}$ is birational by Proposition~\ref{prop.sarkisov}. Moreover, if $\hat{q}\geq 8$, then $\Cl(\hat{X})$ is torsion-free by Proposition~\ref{p.torsion}, so $e=d\geq 2$ and $\dim|s_3A_{\hat{X}}|\geq 0$ if $s_3\not=0$ by Proposition~\ref{prop.sarkisov}. Combining these facts with \eqref{e.k3index11qQ8} and Table~\ref{table.qgeq6} yields the following possibilities for $(\hat{q},e)$:
	\begin{itemize}
		\item $\hat{q}=5$, $e=1$;
		
		\item $\hat{q}=7$, $e=3$;
		
		\item $\hat{q}=17$, $e=5$.
	\end{itemize}
	In the following, each case will be considered individually. Moreover, we also need the following two additional equalities, which are obtained by applying \eqref{eq.key} to $k=2$ and $k=4$, respectively:
    \begin{align}
        \label{e.k2index11qQ8}
		\hat{q}=4s_2+(4m_2+1)e,\\
        \label{e.k4index11qQ8}
		\hat{q}=2s_4+(2m_4+1)e.
    \end{align}
	
	
	\subsubsection{The case $\boldsymbol{\hat{q}=5}$ and $\boldsymbol{e=1}$}
	
	As $d\geq 2$, the group $\Cl(\hat{X})$ is not torsion-free by Proposition~\ref{prop.sarkisov}. Moreover, note that $1\leq s_4\leq 2$ by \eqref{e.k4index11qQ8} and $\dim \hat{\sS}_4=1$, so it follows from~\eqref{i2.p-sarkisov} and Table~\ref{table.qgeq5} that $s_4=2$ and the numerical type of $\hat X$ appears as one of \textnumero\,4 and \textnumero\,5 in Table~\ref{table.qgeq5}. In particular, we have $d=2$ and $S_2=F$. This implies $s_3=m_2=1$ and $m_3=s_2=0$ by Proposition~\ref{prop.sarkisov}, \eqref{e.k3index11qQ8} and \eqref{e.k2index11qQ8}. 
	As $m_4=0$ by \eqref{e.k4index11qQ8}, we obtain $\beta_4=6/11$; as $m_2=1$,  we obtain $\beta_2=14/11$.
	Since $2S_2\in \sS_4$, we have 
    \[
    2\bar F+2\bar E=2\bar S_2+(2\beta_2-\beta_4)\bar{E}\in \bar\sS_4.
    \]
    Pushing it forward to $\hat X$ yields $2\hat{E}\in \hat{\sS}_4$. This implies $2\hat{S}_3 \sim 2A_{\hat{X}}\sim 2\hat{E} \in \hat{\sS}_4$ as $\hat{S}_3\sim_{\bQ} A_{\hat{X}}\sim_{\bQ} \hat{E}$ and $d=|\Cl(\hat{X})_t|=2$. Then we get $|2A_{\hat X}|=\hat{\sS}_4$ as $\dim |2A_{\hat X}|=\dim \hat{\sS}_4=1$, and hence $2\hat{S}_3\in \hat\sS_4$, 
	which contradicts Lemma~\ref{l.elementsinMqQ=8} as $\hat{S}_3\not=\hat{E}$.
	
	
	\subsubsection{The case $\boldsymbol{\hat{q}=7}$, $\boldsymbol{e=3}$}

        In this case, we have $s_4=2$ by \eqref{e.k4index11qQ8}.
	We also have $s_3=m_3=0$ by \eqref{e.k3index11qQ8} and $S_3=F$ by Proposition~\ref{prop.sarkisov}. In particular, we get $d=e=3$. Thus $s_2>0$ and $\Cl(\hat{X})$ is torsion-free by Proposition~\ref{prop.sarkisov}. Then we obtain $s_2=1$ by \eqref{e.k2index11qQ8}. As $\beta_6=9/11+m_6$, applying \eqref{eq.key} to $k=6$ yields $3m_6+s_6=3$. As $\dim\sS_6=3$ and $\hat{f}$ is birational, it follows from Proposition~\ref{prop.sarkisov} that $s_6>0$. Therefore, $s_6=3$ and $m_6=0$. In particular, since $\Cl(\hat{X})$ is torsion-free, we get $\dim|3A_{\hat{X}}|\geq \dim\hat{\sS}_6=3$, and hence the numerical type of $\hat{X}$ appears as one of \textnumero\,7--9 in Table~\ref{table.qgeq6}. In particular, we have $\dim|-K_{\hat{X}}|\geq 15$, and it follows from \cite[Theorem 1.1]{Prokhorov2016} that $\hat{X}$ is isomorphic to one of the following:
	\[
	\bP(1^2,2,3),\quad X_6\subset \bP(1,2^2,3,5),\quad X_6\subset \bP(1,2,3^2,4).
	\]
	Now we divide the proof into three subcases according to the type of $\hat{X}$.
	
	
	\paragraph{\textit{The subcase where $\hat{X}\cong \bP(1^2,2,3)$}}
	As $\dim|A_{\hat{X}}|=1$ and $e=3$, there exists a prime divisor $\hat{D}\in |A_{\hat{X}}|$ which is different from $\hat{S}_2$ and $\hat{E}$. Note that we have 
	\[
	3=\dim |2A_{\hat{X}}|=\dim|3A_{\hat{X}}|-\dim \hat{\sS}_6.
	\]
    In particular, since both $\hat D+|2A_{\hat{X}}|$ and $\hat{\sS}_6$ are sublinear systems of $|3A_{\hat{X}}|$, there must exist an element $\hat{\Delta}\in |2A_{\hat{X}}|$ such that $\hat{D}+\hat{\Delta}\in \hat{\sS}_6$. Then there exists a unique rational number $\gamma$ such that $\hat{f}^*(\hat{D}+\hat{\Delta})-\gamma \bar{F}$ is an effective integral divisor contained in $\bar{\sS}_6$. Denote by $D$ and $\Delta$ the strict transforms of $\hat{D}$ and $\hat{\Delta}$ on~$X$, respectively. Then $D$ is different from $S_2$ and $S_3(=F)$. In particular, we have $D\sim a A$ for some integer $a\geq 4$ because $\Cl(X)$ is torsion-free.  Pushing the integral divisor $\hat{f}^*(\hat{D}+\hat{\Delta})-\gamma \bar{F}$ forward to $X$ shows that $D+\Delta +\delta F\in \sS_6$ for some $\delta\in \bZ_{\geq 0}$. As $\hat{E}\not=\hat{\Delta}$, we have $\Delta\not=0$. However, as $d=3$ and $a\geq 4$, we must have $\delta=0$, $\Delta=S_2$ and $D\in |4A|$. So $\hat \Delta=\hat{S}_2\in |2A_{\hat{X}}|$, which is impossible as $s_2=1$.
	
	
	\paragraph{\textit{The subcase where $\hat{X}\cong X_6\subset \bP(1,2^2,3,5)$}}
	In this case, we have
	\begin{center}
		$\dim|3A_{\hat{X}}|=\dim\sS_6=3$\quad and \quad $2=\dim|2A_{\hat{X}}|>\dim\sS_4=1$.
	\end{center}
	So $\hat{\sS}_6=|3A_{\hat{X}}|$. Let $\bar{\sN}$ be the sublinear system of $\bar{\sS}_6$ such that $\hat{f}_*\bar{\sN}$ coincides with $\hat{S}_2+|2A_{\hat{X}}|$ as sublinear systems of $|3 A_{\hat{X}}|=\hat{\sS}_6$. Denote by $\sN$ the push-forward of $\bar{\sN}$ to $X$. Since $S_2$ is contained in the fixed part of $\sN$, we get $\sN'\coloneqq \sN-S_2 \subset \sS_4$, which is impossible as $\dim\sN'=\dim\sN=\dim|2A_{\hat{X}}|=2$ and $\dim\sS_4=1$.
 
	
	\paragraph{\textit{The subcase where $\hat{X}\cong X_6\subset \bP(1,2,3^2,4)$}} Recall from \cite[Section~4.9]{Prokhorov2022} that we can write
        \[
	K_{\bar{X}} \sim_{\bQ} \hat{f}^*K_{\hat{X}} + c\bar{F},\quad \bar{\sS}_k \sim_{\bQ} \bar{f}^*\hat{\sS}_k-\gamma_k\bar{F},\quad \bar{E} \sim_{\bQ} \hat{f}^*\hat{E}-\delta\bar{F}
	\]
	with $c\in \bQ_{>0}$ and $\gamma_k,\delta\in \bQ_{\geq 0}$. For any $k\geq 1$ and $k\not=3$, it follows from \cite[(4.5)]{Prokhorov2022} that we have
	\begin{equation}
		\label{e.gammakqQ=8}
		-8s_k+7k=3(cs_k-7\gamma_k).
	\end{equation}
	Then applying \eqref{e.gammakqQ=8} to $k=2$, $4$ and $6$ yields
	\[
	\gamma_2=\frac{c-2}{7},\quad \gamma_4=2\gamma_2,\quad \gamma_6=3\gamma_2.
	\]
	In particular, we get $c\geq 2$, and thus $\hat{f}(\bar{F})$ must be a point as $c=1$ if $\hat{f}(\bar{F})$ is a curve. If $c>2$, then $\gamma_4>0$ and $\gamma_6>0$. In particular, as $|2A_{\hat{X}}|=\hat{\sS}_4$ and $|3A_{\hat{X}}|=\hat{\sS}_6$, the point $\hat{f}(\bar{F})$ is contained in the base loci of $|2A_{\hat{X}}|$ and $|3A_{\hat{X}}|$. So we obtain
	\[
	\hat{f}(\bar{F})=[0:0:0:0:1] \in \hat{X}=X_6\subset \bP(1,2,3^2,4).
	\]
	Therefore, $\hat{f}(\bar{F})$ is a point of local index four, and thus $c=1/4$ by \cite[Remark 5.1]{Prokhorov2016}, which gives a contradiction. So we have $c=2$ and then $\gamma_4=0$. Since $|2A_{\hat{X}}|=\hat{\sS}_4$ is a pencil, there exists an element $\hat{D}\in \hat{\sS}_4$ containing $\hat{f}(\bar{F})$. In particular, as $\gamma_4=0$, there exists a positive integer $a$ such that
	\[
	\hat{f}^*\hat{D}=\bar{D}+a\bar{F}\in\bar{\sS}_4 \sim_{\bQ} \hat{f}^*\hat{\sS}_4= \hat{f}^*|2A_{\hat{X}}|,
	\]
	where $\bar{D}$ is the strict transform of $\hat{D}$. This contradicts Lemma~\ref{l.elementsinMqQ=8} as $e=d=3$. 
	
	
	\subsubsection{The case $\boldsymbol{\hat{q}=17}$, $\boldsymbol{e=5}$} 
	
	In this case, we have $s_4=6$ and $\dim|6A_{\hat{X}}|=\dim\hat{\sS}_4=1$ by \eqref{e.k4index11qQ8} and Table~\ref{table.qgeq6}. So $\Cl(X)$ is torsion-free by Proposition~\ref{p.torsion}, and so $|6A_{\hat{X}}|=\hat{\sS}_4$. Let $D$ be the unique element in $|2A_{\hat{X}}|$. Then $D$ is a prime divisor and $3D\in \hat{\sS}_4$. This contradicts Lemma~\ref{l.elementsinMqQ=8} as $e=5$.

\renewcommand\thesection{\Alph{section}}
\setcounter{section}{1}	
	
	\section*{Appendix. Database for \texorpdfstring{$\boldsymbol{\bQ}$}{Q}-Fano threefolds with large \texorpdfstring{$\boldsymbol{q\bQ}$}{qQ}}
        \addcontentsline{toc}{section}{Appendix. Database for \texorpdfstring{$\boldsymbol{\bQ}$}{Q}-Fano threefolds with large \texorpdfstring{$\boldsymbol{q\bQ}$}{qQ}}
\setcounter{subsection}{0}      
        
	\subsection{Data for $\boldsymbol{q\bQ(X)\geq 6}$}
	
	We collect in Table~\ref{table.qgeq6} below the possible numerical invariants for terminal Fano threefold with $q\bQ(X)\geq 6$ and $b_X<3$. It can be obtained by using the same computer program as that in the proof of Lemma~\ref{l.Fanothreefoldscandidates} or the \textsc{Grdb} \cite{BrownKasprzyk2009}. We denote by $A$ a Weil divisor such that $-K_X\sim q\bQ(X) A$. Moreover, we remark that the assumption $b_X<3$ holds automatically if $X$ is a terminal $\bQ$-Fano threefold with $q\bQ(X)\geq 6$ by \cite[Theorem 4.4]{LiuLiu2023}. 
	
	Assume in addition that $X$ is a terminal $\bQ$-Fano threefold. Then we use the symbol ``$+$'' as a superscript of its numbering \textnumero\, if it can be geometrically realised by appropriate examples, ``$+!$'' if it is completely described, ``$-$'' if it cannot occur and ``$?$'' if it is unknown (see \cite{BrownSuzuki2007a,Prokhorov2010,Prokhorov2013,Prokhorov2016,Prokhorov2022a} for the details). The symbol ``$\dagger$'' is used as a subscript in the case where $\Cl(X)$ is possibly not torsion-free (see Proposition~\ref{p.torsion}).
	
	\renewcommand*{\arraystretch}{1.6}
	\begin{longtable}{|M{0.7cm}|M{0.4cm}|M{1.8cm}|M{0.7cm}|M{0.7cm}|M{1.2cm}|M{0.61cm}|M{0.61cm}|M{0.61cm}|M{0.61cm}|M{0.61cm}|M{0.61cm}|M{0.61cm}|M{0.61cm}|M{0.61cm}|}
		\caption{Data for $q=qW(X)=q\bQ(X)\geq 6$}
		\label{table.qgeq6}
		\\
		\hline
		
		\multirow{2}{*}{\textnumero}
		& \multirow{2}{*}{$q$} 
		& \multirow{2}{*}{$\cR_X$} 
		& \multirow{2}{*}{$c_1^3$} 
		& \multirow{2}{*}{$c_2c_1$} 
		& \multirow{2}{*}{$b_X\approx$} 
		& \multicolumn{9}{c|}{$\dim|kA|$}
		\\
		\cline{7-15}

		&
		&
		&
		&
		&
		& $|A|$
		& $|2A|$
		& $|3A|$
		& $|4A|$
		& $|5A|$
		& $|6A|$
		& $|7A|$
		& $|8A|$
		& $|9A|$
		\\
		\hline
		
		1$^{+!}$
		& $6$
		& $\{5\}$
		& $\frac{216}{5}$
		& $\frac{96}{5}$
		& $2.25$
		& $1$
		& $3$
		& $6$
		& $10$
		& $16$
		& $23$
		& $32$
		& $43$
		& $56$
		\\
		\hline
		
		2$^?$
		& $6$
		& $\{5,7\}$
		& $\frac{648}{35}$
		& $\frac{432}{35}$
		& $1.5$
		& $0$
		& $1$
		& $2$
		& $4$
		& $7$
		& $10$
		& $14$
		& $18$
		& $24$
		\\
		\hline
		
		3$^?$
		& $6$
		& $\{5,17\}$
		& $\frac{432}{85}$
		& $\frac{192}{85}$
		& $2.25$
		& $0$
		& $0$
		& $0$
		& $0$
		& $1$
		& $2$
		& $3$
		& $4$
		& $5$
		\\
		\hline
		
		4$^?$
		& $6$
		& $\{5,11\}$
		& $\frac{216}{55}$
		& $\frac{456}{55}$
		& $0.4736$
		& $0$
		& $0$
		& $0$
		& $0$
		& $1$
		& $2$
		& $3$
		& $4$
		& $5$
		\\
		\hline
		
		5$^?$
		& $6$
		& $\{5,7^2\}$
		& $\frac{432}{35}$
		& $\frac{192}{35}$
		& $2.25$
		& $-1$
		& $0$
		& $1$
		& $2$
		& $4$
		& $6$
		& $9$
		& $11$
		& $15$
		\\
		\hline
		
		6$^?$
		& $6$
		& $\{7,11\}$
		& $\frac{432}{77}$
		& $\frac{480}{77}$
		& $0.9$
		& $-1$
		& $0$
		& $0$
		& $1$
		& $1$
		& $3$
		& $4$
		& $5$
		& $7$
		\\
		\hline
		
		7$^{+!}$
		& $7$
		& $\{2,3\}$
		& $\frac{343}{6}$
		& $\frac{119}{6}$
		& $2.8823$
		& $1$
		& $3$
		& $6$
		& $10$
		& $15$
		& $22$
		& $30$
		& $40$
		& $52$
		\\
		\hline
		
		8$^{+!}$
		& $7$
		& $\{2^3,5\}$
		& $\frac{343}{10}$
		& $\frac{147}{10}$
		& $2.3333$
		& $0$
		& $2$
		& $3$
		& $6$
		& $9$
		& $13$
		& $18$
		& $24$
		& $31$
		\\
		\hline
		
		9$^{+!}$
		& $7$
		& $\{2,3^2,4\}$
		& $\frac{343}{12}$
		& $\frac{161}{12}$
		& $2.1304$
		& $0$
		& $1$
		& $3$
		& $5$
		& $7$
		& $11$
		& $15$
		& $20$
		& $26$
		\\
		\hline
		
		10$^+$
		& $7$
		& $\{2^2,3,5\}$
		& $\frac{343}{15}$
		& $\frac{203}{15}$
		& $1.6896$
		& $0$
		& $1$
		& $2$
		& $4$
		& $6$
		& $9$
		& $12$
		& $16$
		& $21$
		\\
		\hline
		
		11$^?$
		& $7$
		& $\{3,6,9\}$
		& $\frac{343}{18}$
		& $\frac{119}{18}$
		& $2.8823$
		& $0$
		& $0$
		& $1$
		& $2$
		& $4$
		& $7$
		& $10$
		& $13$
		& $17$
		\\
		\hline
		
		12$^?_\dagger$
		& $7$
		& $\{2,6,10\}$
		& $\frac{343}{30}$
		& $\frac{203}{30}$
		& $1.6896$
		& $0$
		& $0$
		& $0$
		& $1$
		& $2$
		& $4$
		& $6$
		& $8$
		& $10$
		\\
		\hline
		
		13$^?$
		& $7$
		& $\{2,3,13\}$
		& $\frac{343}{78}$
		& $\frac{539}{78}$
		& $0.6363$
		& $0$
		& $0$
		& $0$
		& $0$
		& $0$
		& $1$
		& $2$
		& $3$
		& $4$
		\\
		\hline
		
		14$^?_\dagger$
		& $7$
		& $\{2^2,3,4,8\}$
		& $\frac{343}{24}$
		& $\frac{161}{24}$
		& $2.1304$
		& $-1$
		& $0$
		& $1$
		& $2$
		& $3$
		& $5$
		& $7$
		& $10$
		& $12$
		\\
		\hline
		
		15$^?$
		& $7$
		& $\{2^2,3,11\}$
		& $\frac{343}{33}$
		& $\frac{245}{33}$
		& $1.4$
		& $-1$
		& $0$
		& $1$
		& $1$
		& $2$
		& $4$
		& $5$
		& $7$
		& $9$
		\\
		\hline
		
		16$^+$
		& $7$
		& $\{2^3,3,4,5\}$
		& $\frac{343}{60}$
		& $\frac{497}{60}$
		& $0.6901$
		& $-1$
		& $0$
		& $0$
		& $1$
		& $1$
		& $2$
		& $3$
		& $4$
		& $5$
		\\
		\hline
		
		17$^?$
		& $7$
		& $\{2^3,5,8\}$
		& $\frac{343}{40}$
		& $\frac{273}{40}$
		& $1.2564$
		& $-1$
		& $0$
		& $0$
		& $1$
		& $2$
		& $3$
		& $4$
		& $6$
		& $7$
		\\
		\hline
		
		18$^?$
		& $7$
		& $\{3,8,9\}$
		& $\frac{343}{72}$
		& $\frac{329}{72}$
		& $1.0425$
		& $-1$
		& $-1$
		& $0$
		& $0$
		& $1$
		& $1$
		& $2$
		& $3$
		& $4$
		\\
		\hline
		
		19$^{+!}$
		& $8$
		& $\{3^2,5\}$
		& $\frac{512}{15}$
		& $\frac{208}{15}$
		& $2.4615$
		& $0$
		& $1$
		& $3$
		& $4$
		& $7$
		& $10$
		& $13$
		& $18$
		& $23$
		\\
		\hline
		
		20$^{+!}$
		& $8$
		& $\{3,7\}$
		& $\frac{512}{21}$
		& $\frac{304}{21}$
		& $1.6842$
		& $0$
		& $1$
		& $2$
		& $3$
		& $5$
		& $7$
		& $10$
		& $13$
		& $17$
		\\
		\hline
		
		21$^+$
		& $8$
		& $\{5,7\}$
		& $\frac{512}{35}$
		& $\frac{432}{35}$
		& $1.1851$
		& $0$
		& $0$
		& $1$
		& $2$
		& $3$
		& $4$
		& $6$
		& $8$
		& $10$
		\\
		\hline

		22$^-$
		& $8$
		& $\{3,5,11\}$
		& $\frac{2048}{165}$
		& $\frac{928}{165}$
		& $2.2068$
		& $-1$
		& $0$
		& $0$
		& $1$
		& $2$
		& $3$
		& $4$
		& $6$
		& $8$
		\\
		\hline
		
		23$^{+!}$
		& $9$
		& $\{2,4,5\}$
		& $\frac{729}{20}$
		& $\frac{279}{20}$
		& $2.6129$
		& $0$
		& $1$
		& $2$
		& $4$
		& $6$
		& $8$
		& $11$
		& $15$
		& $19$
		\\
		\hline
		
		24$^+$
		& $9$
		& $\{2^3,5,7\}$
		& $\frac{729}{70}$
		& $\frac{549}{70}$
		& $1.3278$
		& $-1$
		& $0$
		& $0$
		& $1$
		& $1$
		& $2$
		& $3$
		& $4$
		& $5$
		\\
		\hline
		
		25$^{+!}$
		& $11$
		& $\{2,3,5\}$
		& $\frac{1331}{30}$
		& $\frac{451}{30}$
		& $2.9512$
		& $0$
		& $1$
		& $2$
		& $3$
		& $5$
		& $7$
		& $9$
		& $12$
		& $15$
		\\
		\hline
		
		26$^+$
		& $11$
		& $\{2,5,7\}$
		& $\frac{1331}{70}$
		& $\frac{759}{70}$
		& $1.7536$
		& $0$
		& $0$
		& $0$
		& $1$
		& $2$
		& $3$
		& $4$
		& $5$
		& $6$
		\\
		\hline
		
		27$^+$
		& $11$
		& $\{2^2,3,4,7\}$
		& $\frac{1331}{84}$
		& $\frac{649}{84}$
		& $2.0508$
		& $-1$
		& $0$
		& $0$
		& $1$
		& $1$
		& $2$
		& $3$
		& $4$
		& $5$
		\\
		\hline
		
		28$^{+!}$
		& $13$
		& $\{3,4,5\}$
		& $\frac{2197}{60}$
		& $\frac{767}{60}$
		& $2.8644$
		& $0$
		& $0$
		& $1$
		& $2$
		& $3$
		& $4$
		& $5$
		& $7$
		& $9$
		\\
		\hline
		
		29$^{+!}$
		& $13$
		& $\{2,3^2,5,7\}$
		& $\frac{2197}{210}$
		& $\frac{1157}{210}$
		& $1.8988$
		& $-1$
		& $-1$
		& $0$
		& $0$
		& $0$
		& $1$
		& $1$
		& $1$
		& $2$
		\\
		\hline
		
		30$^{+!}$
		& $17$
		& $\{2,3,5,7\}$
		& $\frac{4913}{210}$
		& $\frac{1717}{210}$
		& $2.8613$
		& $-1$
		& $0$
		& $0$
		& $0$
		& $1$
		& $1$
		& $2$
		& $2$
		& $3$
		\\
		\hline
		
		31$^{+!}$
		& $19$
		& $\{3,4,5,7\}$
		& $\frac{6859}{420}$
		& $\frac{2489}{420}$
		& $2.7557$
		& $-1$
		& $-1$
		& $0$
		& $0$
		& $0$
		& $0$
		& $1$
		& $1$
		& $1$
		\\
		\hline
	\end{longtable}
        
	\subsection{Data for $\boldsymbol{qW(X)\not=q\bQ(X)\geq 3}$}
	
	We collect in Table~\ref{table.qWnot=qQ=4} below the possible numerical invariants for terminal $\bQ$-Fano threefolds with $qW(X)\not=q\bQ(X)\geq 3$, which is obtained in \cite[Proposition 3.2]{Prokhorov2024}. 

	\renewcommand*{\arraystretch}{1.6}
	\begin{longtable}[c]{|M{0.7cm}|M{1cm}|M{2.5cm}|M{1cm}|M{1cm}|M{1.5cm}|M{3cm}|}
		\caption{Data for $qW(X)\not= q\bQ(X)\geq 3$}
		\label{table.qWnot=qQ=4}
		\\
		\hline
		
		\textnumero
        & $q\bQ$
		& $\cR_X$
		& $c_1^3$
		& $c_2c_1$
		& $b_X\approx$
        & \cite{BrownKasprzyk2009}
		\\
		\hline

        1
        & 3
		& $\{3^4,6\}$
		& $\frac{27}{2}$
		& $\frac{15}{2}$
		& $1.8$
		& \textnumero\,30381
        \\
		\hline

        2
        & 3
		& $\{3^4,5,6\}$
		& $\frac{27}{10}$
		& $\frac{27}{10}$
		& $1$
        & \textnumero\,9014
		\\
		\hline

        3
        & 3
		& $\{2,3^2,12\}$
		& $\frac{27}{4}$
		& $\frac{21}{4}$
		& $1.2857$
        & \textnumero\,19801
		\\
		\hline
		
		4
        & 4
		& $\{2^5,6\}$
		& $\frac{64}{3}$
		& $\frac{32}{3}$
		& $2$
        & \textnumero\,35882
		\\
		\hline
		
		5
        & 4
		& $\{2^5,5,6\}$
		& $\frac{128}{15}$
		& $\frac{88}{15}$
		& $1.4545$
        & \textnumero\,23440
		\\
		\hline
		
		6
        & 4
		& $\{2^5,6,7\}$
		& $\frac{64}{21}$
		& $\frac{80}{21}$
		& $0.8$
        & \textnumero\,10111
		\\
		\hline
		
		7
        & 4
		& $\{2^3,10\}$
		& $\frac{128}{5}$
		& $\frac{48}{5}$
		& $2.6666$
        & \textnumero\,37308
		\\
		\hline
		
		8
        & 4
		& $\{2^3,3,10\}$
		& $\frac{64}{15}$
		& $\frac{104}{15}$
		& $0.6153$
        & \textnumero\,14290
		\\
		\hline

		9
        & 4
		& $\{2^3,5,10\}$
		& $\frac{64}{5}$
		& $\frac{24}{5}$
		& $2.6666$
        & \textnumero\,29220
		\\
		\hline
		
		10
        & 4
		& $\{2^3,7,10\}$
		& $\frac{256}{35}$
		& $\frac{96}{35}$
		& $2.6666$
        & \textnumero\,20313
		\\
		\hline
	\end{longtable}


\end{document}